\documentclass{amsart}
\usepackage[utf8]{inputenc}

\title{Metric Lines in Engel-type Groups}
 \author[Bravo-Doddoli]{Alejandro Bravo-Doddoli}
  \address{Alejandro Bravo-Doddoli;
 {University of Michigan, 530 Church St, Ann Arbor, MI 48109, United States
 \\
\href{abravodo@umich.edu}{abravodo@umich.edu}}
 }


\date{September 2026}
\usepackage[english]{babel}
\usepackage[utf8]{inputenc}
\usepackage{libertine}
\usepackage{graphicx}
\usepackage{floatflt}
\usepackage{blindtext}
\usepackage{enumitem}
\usepackage{amsthm}
\usepackage{subfig}
\usepackage{listings}
\usepackage{listingsutf8}
\usepackage{amsmath}
\usepackage{framed}
\usepackage{minibox}
\usepackage{float}
\usepackage{wrapfig}
\usepackage{longtable}
\usepackage[strict]{changepage}
\usepackage{pgfplots}
\usepackage{nicefrac}
\usepackage{units}
\usepackage{etoolbox,tikz}
\usepackage{bbm}
\usepackage{hyperref}

\usepackage[normalem]{ulem} 

\usepackage[autostyle]{csquotes}

\usepackage[all]{xy}
\usepackage{filecontents}

\usetikzlibrary{external}
\usetikzlibrary{cd}

\AtBeginEnvironment{tikzcd}{\tikzexternaldisable}
\AtEndEnvironment{tikzcd}{\tikzexternalenable}

\usepackage{amsmath}
\usepackage{amsfonts}

\usepackage{amsthm}
\usepackage {amsmath, amssymb}
\usepackage{pb-diagram}
\usepackage{tikz-cd}
\usepackage{mathtools}
\usepackage{amssymb}

\usetikzlibrary{matrix}
\pgfplotsset{width=11cm,compat=1.9}
\usepgfplotslibrary{external}
\newtheorem{Theorem}{Theorem}[section]
\newtheorem{defi}[Theorem]{Definition}
\newtheorem{Prop}[Theorem]{Proposition}
\newtheorem{conj}[Theorem]{Conjecture}
\newtheorem{Cor}[Theorem]{Corollary}
\newtheorem{lemma}[Theorem]{Lemma}
\newtheorem{Remark}[Theorem]{Remark}
\theoremstyle{plain}
\numberwithin{figure}{section}

\def\Na{\mathbb{N}}
\def\R{\mathbb{R}}
\newcommand{\Lag}{\mathfrak{g}}
\newcommand{\Laa}{\mathfrak{a}}
\newcommand{\Lav}{\mathfrak{v}}
\newcommand{\Lah}{\mathfrak{h}}
\newcommand{\Ag}{\mathbb{A}}
\newcommand{\G}{\mathbb{G}}
\newcommand{\Di}{\mathcal{D}}
\newcommand{\Ho}{\mathcal{H}}
\newcommand{\Vs}{\mathcal{V}}
\newcommand{\T}{\mathcal{T}}
 
\def\R{\mathbb{R}}
\def\D{\mathcal{D}}
\def\MS{\mathrm{M}_F}

\def\ma{metabelian }
\def\Ri{Riemannian }
\def\Ri{Riemannian }
\def\SR{Sub-Riemannian } 
\def\J{\mathcal{J}^k(\mathbb{R},\mathbb{R})}
\def\sR{sub-Riemannian }

\DeclareMathOperator{\Eng}{Eng}
\DeclareMathOperator{\V}{V}

\begin{document}

\begin{abstract}
In the framework of sub-Riemannian manifolds, a natural question is: What are the \enquote{metric lines} (i.e., isometric embeddings of the real line)? This article presents a conjecture classifying metric lines in Carnot groups and investigates the question for Carnot groups of arbitrary rank. It focuses on metric lines in Engel-type groups, whose sub-Riemannian structure is defined on a distribution of rank $n+1$. The approach employs a novel technique, the sequence method, initially developed to study metric lines in jet space.
\end{abstract}

\maketitle

\tableofcontents


\addtocontents{toc}{\protect\setcounter{tocdepth}{1}}
\section*{Declarations}

\subsection*{Ethical Approval.}
Not applicable.

\subsection*{Funding }
Not applicable.
\addtocontents{toc}{\protect\setcounter{tocdepth}{2}}

\section{Introduction}\label{sec:int}

This work proposes a conjecture that classifies metric lines in Carnot groups within the framework of sub-Riemannian geometry. We make initial progress toward addressing this conjecture for metabelian Carnot groups with a semidirect product structure. Since every Carnot group admits a sub-Riemannian structure, we broaden the context by considering a sub-Riemannian manifold $M$ and a sub-Riemannian geodesic $\gamma $ in $ M$, defined as a locally arc-length-minimizing curve. A central question arises: under what conditions is a sub-Riemannian geodesic a global minimizer? A curve $\gamma: \R \to M$ is a metric line if $\gamma$ is a globally minimizing geodesic. The conjecture classifying metric lines in Carnot groups is stated as follows:

\begin{conj}\label{conj}
    Let $\G$ be a Carnot group with left-invariant \sR structure. The metric lines in $\G$ are precisely the geodesics $\gamma$ parameterized by arc length satisfying the following condition: there exists a unit vector $v\in\Lag$ such that
\begin{equation*}
   v =  \lim_{t \to - \infty} (L_{\gamma^{-1}(t)})_* \dot{\gamma} = \lim_{t \to  \infty} (L_{\gamma^{-1}(t)})_* \dot{\gamma}, 
\end{equation*}
where $L_g$ is the left translation by the element $g\in \G$, and $(L_g)_*$ is the push-forward of $L_g$.
\end{conj}

The first part of this paper introduces the main contribution, Theorem \ref{the:metrtic-lines-method}, which presents a novel approach for verifying that a specific geodesic is a metric line in a metabelian Carnot group with a semidirect product structure. This approach is called the sequence method. A group $\G$ is \ma if its commutator subgroup $[\G,\G]$ is abelian. In previous work \cite{BravoDoddoli2024}, we performed the symplectic reduction of the geodesic flow on metabelian nilpotent groups.
Every Carnot group is nilpotent. In this work, we use this reduction to classify geodesics satisfying the condition in Conjecture \ref{conj}. Based on this classification, we select a geodesic $\gamma \in \G$ as a candidate metric line.

We then construct a sub-Riemannian manifold $\MS$, called the magnetic space, along with a projection $\pi_F:\G \to \MS$ such that $c := \pi_F\circ \gamma$ is a geodesic in $\MS$. The projection $\pi_F$ admits horizontal lifts of curves. Furthermore, Lemma \ref{lem:sub-submersion} states that if $\gamma$ is the horizontal lift of a metric line $c $ in $ \MS$ via the projection $\pi_F$, then $\gamma$ is a metric line. Consequently, studying metric lines in $\G$ reduces to studying metric lines in $\MS$. Theorem \ref{the:metrtic-lines-method} provides sufficient conditions for the geodesic $c$ to be a metric line.

The second part of this paper applies the sequence method to establish Conjecture \ref{conj} in the Engel-type group $\Eng(n)$. According to this classification, geodesics on $\Eng(n)$ are of the following types: line, small-oscillation, r-periodic, and r-homoclinic; see Section \ref{subsub:clas-sr-geo}.  Theorem \ref{the:main-1} presents the second principal result and offers a comprehensive classification of metric lines in the Engel-type group.

\begin{Theorem}\label{the:main-1}
The metric lines in $\Eng(n)$ are precisely the line and \(r\)-homoclinic geodesics.
\end{Theorem}
Proposition \ref{prp:no-met-lin-Eng} shows that small-oscillation and $r$-periodic geodesics are not metric lines. Furthermore, Lemma \ref{lemma:abn-geo-Eng} shows that the abnormal geodesics are line-geodesics and, consequently, metric lines.

\begin{figure}%
    \centering
    {{\includegraphics[width=5cm]{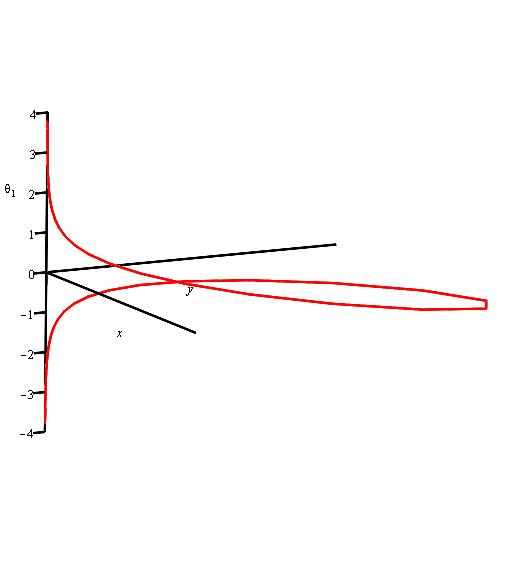} }}
    \caption{The images show the space $\R^{3}$, with coordinates $(x_1,x_2,\theta_0)$,  and the projection under \(\pi\) of the homoclinic geodesic \(\gamma\).} 
    \label{fig:euler}
\end{figure}

\subsubsection{The Euler-Elastica problem}

The Euler-Elastica problem defines a remarkable family of curves (elastica is a Latin word for a thin strip of elastic material) \cite{levien2008elastica,kot2014first,jurdjevic1997geometric}. There are several ways to set up the Euler-Elastica problem. Of particular interest in this work is the following characterization: a plane curve $\eta$ solves the Euler-Elastica problem if, at each point, its curvature $\kappa$ is proportional to the distance from $\eta$ to a given straight line called the directrix. Within this family, the Euler-soliton stands out as a distinguished curve. 

The classical Engel group, denoted by $\Eng$, is the Engel-type group $\Eng(1)$. Geodesics on the Engel group are characterized as follows. Consider the canonical projection $\pi:\Eng \to \R^2$, defined in Section \ref{sec:preliminary}. If $\gamma$ is a geodesic in $\Eng$, then $\pi\circ \gamma$ solves the Euler-Elastica problem. Conversely, if $\gamma$ is the horizontal lift via $\pi$ of a solution to the Euler-Elastica problem, then $\gamma$ is a geodesic. In \cite{ardentov2013conjugate,ardentov2011extremal,ardentov2015cut,ardentov2017maxwell},
A. Ardentov and Y. Sachkov showed that, up to Carnot dilation and translation, there exists only one metric line in the Engel group. Furthermore, the projection of this metric line onto the plane is the distinguished Euler-soliton. We extend this result as follows. Consider the canonical projection $\pi:\Eng(n) \to \R^{n+1}$. Proposition \ref{prp:Euler-Elastica-Eng} proves that if $\gamma$ is a geodesic in $\Eng(n)$ with zero angular momentum, then $\pi\circ \gamma$ lies in a plane and solves the Euler-Elastica problem. In particular, the projection of an $r$-homoclinic geodesic is the Euler-soliton, as shown in Figures \ref{fig:euler} and \ref{fig:perio-hom-curve}.

\begin{figure}%
    \centering
    {{\includegraphics[width=2.5cm]{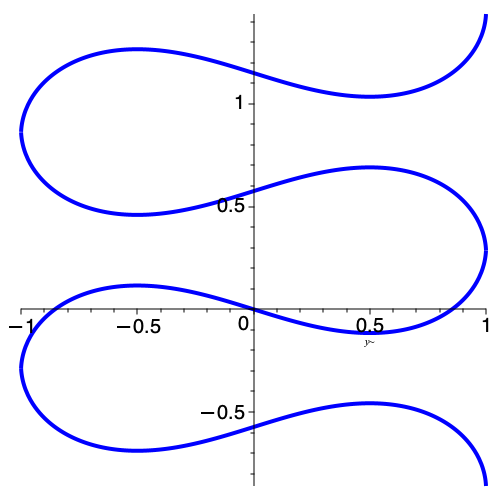} }}
    \qquad
    {{\includegraphics[width=2.5cm]{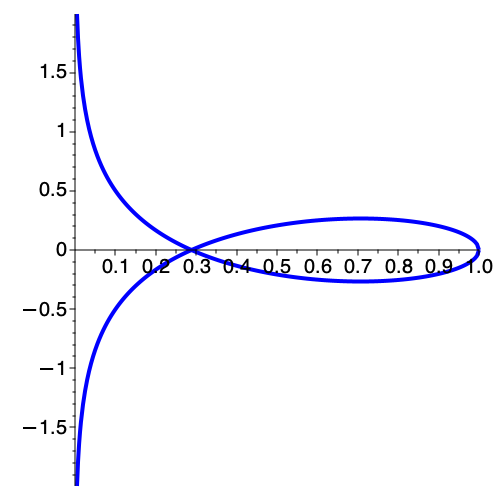} }}%
    \caption{The images show two solutions to the Euler-Elastica problem in the $(x,\theta_0)$-plane. The first panel presents a generic solution, and the second displays the Euler-soliton solution.}    
    \label{fig:perio-hom-curve}
\end{figure}

\addtocontents{toc}{\protect\setcounter{tocdepth}{1}}
\subsubsection*{Acknowledgment}
\addtocontents{toc}{\protect\setcounter{tocdepth}{2}}

I thank my Ph.D. advisor, Richard Montgomery, who introduced me to sub-Riemannian geometry, Carnot groups, and metric lines, and whose invaluable help has continued since my time at UCSC. I thank my mentor, Anthony Bloch, and the mathematics department at the University of Michigan, where I conducted this work, for supporting my application and welcoming me into the department. 

I thank Enrico Le Donne for inviting me to the University of Fribourg and introducing me to metabelian groups; Nicola Paddeu and Gil Bor for email conversations about this work; and Andrei Ardentov and Yuri Sachkov, whose work inspired me to write this paper.

\section{Preliminaries}\label{sec:preliminary}

\subsection{\SR Geometry and Carnot Groups}

A \sR manifold consists of a smooth manifold $M$, a bracket-generating distribution $\D$, and an inner product $\langle\cdot,\cdot\rangle$ on $\D$. A curve $\gamma$ is called horizontal if it is tangent to $\D$ whenever $\dot{\gamma}(t)$ exists. Chow's theorem states that any two points in $M$ can be connected by a horizontal curve if $\D$ is bracket generating  \cite{jurdjevic1997geometric,agrachev2008geometry}. The \sR arc length of a smooth horizontal curve $\gamma$ is defined as in \Ri geometry, as the integral over the curve's domain of the \sR norm of $\dot{\gamma}(t)$. Similarly, the \sR distance $dist_M(p,q)$ is defined as the infimum of the \sR arc length over all smooth horizontal curves connecting $p$ and $q$ \cite{montgomery2002tour,agrachev2019comprehensive}. In the context of Carnot groups, this metric is called the Carnot–Carathéodory metric \cite{mitchell1985carnot,gromov1996carnot,monti2001surface}.

The Pontryagin maximum principle provides necessary conditions for sub-Riemannian geodesics, that is, locally arc-length-minimizing horizontal curves. The Pontryagin maximum principle distinguishes normal and abnormal extremals \cite{montgomery2002tour,agrachev2019comprehensive}. The \sR kinetic energy $H_{sR}:T^*M \to \R$ is a Hamiltonian function, and the projections to $M$ of solutions to the corresponding Hamiltonian system are normal geodesics. There is no analogous notion of abnormal geodesics in \Ri geometry \cite{liu1995shortest,bryant1993rigidity,montgomery1994abnormal}. The paper's primary goal is to study when a normal geodesic is globally minimizing.

     Let $M$ be a \sR manifold equipped with a \sR distance. A curve $\gamma:\R \to M$ is a metric line if it is a globally minimizing geodesic, i.e.,  
$$|a-b| = dist_{M}(\gamma(a),\gamma(b))\;\;\; \text{for all}\;\; a,b \in \R. $$

Alternative names for the term \enquote{metric line} are \enquote{globally minimizing geodesic}, \enquote{isometric embedding of the real line}, and \enquote{infinite geodesic}; see \cite{agrachev2019comprehensive,hakavuori2023blowups}.

A Carnot group $\G$ is a  connected and simply connected Lie group whose Lie algebra $\Lag$ is stratified and nilpotent, i.e., the Lie algebra $\Lag$ satisfies
$$ \Lag = \V_1 \oplus \dots \oplus \V_s, \;\;\; [\V_j,\V_1] = \V_{j+1},\;\;\;\text{and} \;\;\; \V_{s+1} = \{0\}, $$
where $[\cdot,\cdot]$ is the Lie bracket. The integer $s$ is the step, and the dimension of $\V_1$ is the rank of $\G$ \cite{le2008cornucopia,corwin1990representations}. 

Every Carnot group $\G$ has a natural left-invariant distribution $\D$ defined by $\D_g:= (L_g)_*\V_1$. The distribution $\D$ is bracket generating because $\V_1$ generates $\Lag$. Every Carnot group $\G$ is endowed with a canonical projection $\pi: \G \to  \V:=\G / [\G,\G]$, where $\V$ is a vector space such that $\V \simeq\V_1$. To equip $\G$ with a \sR structure, consider a Euclidean inner product on $\V$. The \sR inner product on $\G$ is then defined as the pull-back by $\pi$ of the Euclidean inner product. Consequently, a horizontal curve $\gamma$ in $\G$ and its projection $\pi\circ\gamma$ have the same arc length. The projection $\pi$ is an example of a \sR submersion, an important concept for this work.


    Let $M$ and $N$ be two \sR manifolds. A submersion $\phi: M \to N$ is called a \sR submersion or submetry if it respects the \sR structures, i.e.,  $\phi_* \D_M = \D_N$ and $\phi^* \langle \cdot,\cdot\rangle_N = \langle \cdot,\cdot\rangle_M$, where $\phi_*$ and $\phi^*$ are the push-forward and pull-back of $\phi$, respectively. 

        The \sR submersion $\phi:M \to N$ admits a horizontal lift of curves. This lift is unique once an initial point is fixed.

Consult \cite{burago2001course,montgomery2002tour} for more details about the horizontal lift. A classical result in metric spaces is the following.

\begin{lemma}[Proposition 1, \cite{bravo2022geodesics}]\label{lem:sub-submersion}
Let $\phi: M \to N$ be a \sR submersion. If $\gamma$ is a horizontal lift of a metric line in $N$ via the map $\phi$, then $\gamma$ is a metric line in $M$. 
\end{lemma}

Let $\G$ be a Carnot group. A geodesic $\gamma$ is called a line-geodesic if $\gamma$ is the horizontal lift of a straight line in $\V$ via the canonical projection $\pi:\G \to \V$. 
Line-geodesics are left translates of horizontal one-parameter subgroups. Lemma \ref{lem:sub-submersion} yields the following result.
\begin{lemma}\label{lem:sR-sub-metric-line}
Every line-geodesic in a Carnot group is a metric line. 
\end{lemma}
The family of metric lines in Carnot groups is nonempty. This raises a natural question: besides line-geodesics, what other metric lines exist in a Carnot group  $\G$? The following proposition characterizes metric lines in Carnot groups.
\begin{Prop}\label{prop:met-char}
        Let $\G$ be a Carnot group with a left-invariant \sR structure. Let $\gamma$ be a geodesic of class $C^1$, parameterized by arc length. A necessary condition for $\gamma$ to be a metric line is that there exists a unit vector $v\in \V_1$ such that 
        \begin{equation}\label{eq:in-met-char}
            1 = \limsup_{t \to \infty} \frac{1}{2t} \int_{-t}^t \langle v,(L_{\gamma^{-1}(s)})_* \dot{\gamma}(s)\rangle ds . 
        \end{equation} 
    \end{Prop}
The proof of Proposition \ref{prop:met-char} is in Section \ref{sec:sequ-meth}.
\begin{Remark}
    A more precise version of Conjecture \ref{conj} requires the vector $v$ to belong to the subspace $\V_1 \subset \Lag$ as in Proposition \ref{prop:met-char}. Since $\V_1$ has not yet been defined at that point, we state Conjecture \ref{conj}  with $v\in\Lag$.
\end{Remark}
\subsubsection{Metabelian Carnot groups}

A group $\G$ is \ma if $[\G,\G]$ is abelian  \cite{robinson2012course}. Every step--3 Carnot group is metabelian. Every  \ma Carnot group $\G$ has a maximal normal abelian subgroup $\Ag$ containing $[\G,\G]$. Consider the action of $\Ag$ on $\G$ by left translation, and define the vector space $\Ho:=  \G/\Ag$, which we identify with a vector subspace of $\V$. Thus, $\G$ is an abelian extension of $\Ho$ by $\Ag$,  a structure that plays a pivotal role in the symplectic reduction of $T^*\G$ as described in \cite{BravoDoddoli2024}. Let $\Laa$ be the Lie algebra of $\Ag$ and define $\Lav:= \V_1\cap \Laa$. The \sR inner product induces a decomposition $\V_1 =  \Lah\oplus \Lav$, where $\Lah$ is the orthogonal complement of $\Lav$. Identifying $\V$ with $T\V$, this decomposition yields an orthogonal splitting $\V =  \Ho\oplus \Vs$, where $\pi_* \Lav = \Vs$ and $\pi_* \Lah = \Ho$. If $\Lah$ is an abelian subalgebra, then  $\Ho$ can be identified with a subgroup of $\G$, and $\G$ admits the semidirect product structure  $\Ho \ltimes \Ag$. 

Let $T^*\G$ be the cotangent bundle of $\G$ with the canonical symplectic structure, and let $H_{sR}:T^*\G \to \R$ be the Hamiltonian function governing the geodesic flow. The Hamiltonian action of $\Ag$ on $T^*\G$ preserves the Hamiltonian $H_{sR}$ and defines a momentum map $J: T^*G \to \Laa^*$ \cite{arnol2013mathematical,ortega2013momentum,marsden1974reduction,marsden1984reduction}. If $(p(t),\gamma(t))$ is a solution of the Hamiltonian system for $H_{sR}$, then the curve $\gamma$ is a geodesic in $\G$. A geodesic $\gamma$ has momentum $\mu\in \Laa^*$ if  $J(p(t),\gamma(t)) = \mu$ for all $t\in\R$.  
By symplectic reduction, the reduced space $T^*G//_{\mu} \Ag$ is symplectomorphic to $T^*\Ho$.  If $\G = \Ho \ltimes \Ag$  is of step $s$, then the reduced Hamiltonian takes the form
\begin{equation}\label{eq:red-Ham}
    H_{\mu}(p_x,x) = \frac{1}{2} \sum_{i=1}^{dim \;\Ho} p_{x_i}^2 +  \frac{1}{2} V_{\mu}(x), \;\;\text{where} \;\; V_{\mu}(x) = \|F_{\mu}(x)\|^2_{\Vs}.
\end{equation}
Here, $\|\cdot\|_{\Vs}$ is the Euclidean norm in $\Vs$ and $F_{\mu}(x)$ is a vector-valued map in $\Vs$ whose entries $F_j(x;\mu):\Ho \to \R$ are polynomials of degree at most $s-1$ depending on the parameter $\mu$. 

\begin{Remark}\label{rem:bij-a-vect-valued}
There is a linear correspondence $\mu\to F_\mu$ between $\Laa^*$ and the corresponding space of vector-valued maps. Hence, prescribing $\mu$ determines $F_{\mu}$. For instance, see Eq. \eqref{eq:F-pol} for an explicit expression for this bijection.  
\end{Remark}

The reduced dynamics occur in a closed set called the Hill region.
\begin{defi}\label{def:Hill-region}
Given a potential polynomial $V_{\mu}(x)$, a connected closed set $\Omega_{\mu}\subseteq\Ho$ is called a Hill set if  $V_{\mu}(x) < 1$  for every $x$ in the interior of $\Omega_{\mu}$ and $V_{\mu}(x) = 1$ for every $x$ on the boundary of $\Omega_{\mu}$. The union of all Hill sets is called the Hill region.    
\end{defi}

The set of pairs $(V_{\mu},\Omega_{\mu})$, where $V_{\mu}(x)$  is in the pencil of potential polynomials and $\Omega_{\mu}$ is a Hill set, determines an isoenergetic algebraic set $\Gamma(\Omega_{\mu})\subset T^*\Ho$ given by
\begin{equation}\label{eq:alg-curv}
    \Gamma(\Omega_{\mu}) := \biggl\{ (p_x,x) \in T^*\Ho :  1 = \sum_{i=1}^{dim \;\Ho} p_{x_i}^2 +   V_{\mu}(x) \;\;\text{and} \;\;x \in \Omega_{\mu}  \biggl\} .
\end{equation}
The Hamilton equations imply that a point in $\Gamma(\Omega_{\mu})$ is an equilibrium point of the reduced system $H_{\mu}$ if and only if the algebraic set $\Gamma(\Omega_{\mu})$ is singular at that point. Thus, classifying the equilibrium points of the reduced dynamics is equivalent to classifying the singularities of $\Gamma(\Omega_{\mu})$.

Symplectic reconstruction is a standard technique for constructing a solution of the unreduced system from a solution of the reduced system \cite{marsden1990reduction}. In the case of metabelian Carnot groups, the symplectic reconstruction provides a three-step prescription for constructing a normal geodesic: 
\begin{itemize}
    \item[1.] Find a solution $(p_x(t),x(t))$ to the reduced Hamiltonian system with an initial condition  in $\Gamma(\Omega_{\mu})$.

   \item[2.] Define a curve $\eta $ in $\V$ as the solution of the system
\begin{equation}\label{eq:pres-geo}
    \dot{\eta}(t) = \sum_{i=1}^{dim \;\Ho} p_{x_i}(t) e_i + \sum_{j=1}^{dim\; \Vs} F_{j}(x(t);\mu) \tilde{e}_{j},
\end{equation} 
    where $\{e_i\}_{i = 1,\dots, dim\; \Ho}$ is a basis for the subspace $\Ho \subset \V$ and $\{\tilde{e}_j\}_{j = 1,\dots, dim\; \Vs}$ is a basis for the subspace $\Vs \subset \V$.

    \item[3.] Let $\gamma$ be the horizontal lift of $\eta$. 
\end{itemize}

The curve $\gamma$ is defined for all $t\in\R$ by the completeness of the Hamiltonian system $H_{\mu}$.

 \begin{Theorem}[\cite{BravoDoddoli2024}]\label{the:background}
The above prescription yields a normal geodesic in $\G$ parameterized by arc length and with momentum $\mu$.  Conversely, this prescription yields every arc-length-parameterized normal geodesic in $\G$ with momentum $\mu$. 
 \end{Theorem}


\subsubsection{Classification of \sR Geodesics}\label{subsubsec:clas-geo-G}

 This section classifies sub-Riemannian geodesics according to their reduced dynamics. The classification aims to identify geodesics satisfying the conditions of Conjecture \ref{conj}.

\textbf{(Regular geodesic)} A geodesic is regular if its reduced dynamics are regular. The reduced dynamics are regular if they are not asymptotic to equilibrium points. A regular geodesic is called bounded if $x(t)$ lies in a compact set of $\Ho$ for $t\in\R$. A regular geodesic is called unbounded if it is not bounded.

\textbf{(Singular geodesic)} A geodesic is singular if its reduced dynamics are singular. The reduced dynamics are singular if they are asymptotic to equilibrium points. If $\gamma$ is a singular geodesic in $\G$, then Eq. \eqref{eq:pres-geo} implies that  $\gamma$ satisfies the condition in Conjecture \ref{conj} if and only if
\begin{equation}\label{eq:asym-pol}
    \lim_{t \to \infty} F_j(x(t);\mu) = \lim_{t \to - \infty}  F_j(x(t);\mu) \qquad \text{for all} \; j = 1, \dots, dim\;\Vs .
\end{equation}

\textbf{(Homoclinic geodesic)} A singular geodesic is homoclinic if its reduced dynamics are homoclinic. The reduced dynamics are homoclinic if 
$$\lim_{t \to \infty} (p_x(t), x(t)) =  (0, x_0) = \lim_{t \to -\infty} (p_x(t), x(t))  .$$ 

\textbf{(Heteroclinic geodesic)} A singular geodesic is heteroclinic if its reduced dynamics are heteroclinic. The reduced dynamics are heteroclinic if 
$$\lim_{t \to \infty} (p_x(t), x(t)) =  (0, x_0), \;\;\text{and}  \lim_{t \to -\infty} (p_x(t), x(t))  =  (0, x_1) , \;\;\text{where}\;x_0 \neq x_1.$$

By definition, homoclinic geodesics satisfy Eq. \eqref{eq:asym-pol}. However, heteroclinic geodesics do not necessarily satisfy Eq. \eqref{eq:asym-pol}. The following definitions distinguish the two cases according to whether a heteroclinic geodesic satisfies this equation.

\textbf{(Direct-type geodesic)} A heteroclinic geodesic is direct-type if its reduced dynamics satisfy Eq. \eqref{eq:asym-pol}.

\textbf{(Turn-back geodesic)} A heteroclinic geodesic is turn-back if its reduced dynamics do not satisfy Eq. \eqref{eq:asym-pol}.

The following corollary states that turn-back geodesics are not metric lines.
\begin{Cor}[Corollary 1.5, \cite{BravoDoddoli2024}]\label{Cor:turn-back}
    Let $\G$ be a \ma Carnot group with a semidirect product structure. Turn-back geodesics are not metric lines. Moreover, bounded geodesics with reduced dynamics within a regular open set of the potential $V_{\mu}(x)$ are not metric lines.  
\end{Cor}
  Lemma \ref{lem:cond-metl-lin-2}, which is a consequence of Proposition \ref{prop:met-char}, provides an alternative proof of the first statement of Corollary \ref{Cor:turn-back}.

\subsubsection{Exponential Coordinates of Second Type}\label{sub-sub-sec:exp-coor}

This section introduces the framework of the sequence method. Consider a metabelian Carnot group $\G$ of step $s$, equipped with a semidirect product structure $\G = \Ho \ltimes \Ag$ such that $dim \;\Lah = n$ and  $dim \; \Lav = 1$.
These conditions imply that the distribution $\D$  has rank $n+1$. The Engel-type group $\Eng(n)$ satisfies these conditions. Under these assumptions, we construct coordinates on $\G$ to define the \sR submersion $\pi_F: \G \to \MS$, where $\MS$ is the sub-Riemannian magnetic space introduced in Section \ref{sec:mag-spa}.

Consider the following basis for $\Lag$:
$$\{E_1,\dots,E_n,E_{0}^{\Laa},E_{1,2}^{\Laa},\dots,E_{n_s,s}^{\Laa}\}.$$
Here $\{E_1,\dots,E_n\}$, $\{E_{0}^{\Laa},E_{1,2}^{\Laa},\dots,E_{n_s,s}^{\Laa}\}$, and $\{E_{1,k}^{\Laa},\dots,E_{n_k,k}^{\Laa}\}$ are bases for $\Lah$, $\Laa$, and $\V_k$ for $1<k\leq s$, respectively. Thus,  $\{E_0^{\Laa}\}$ and $\{E_1,\dots,E_n,E_{0}^{\Laa}\}$ are bases for $\Lav$ and $\V_1$, respectively.

The nilpotency of $\Lag$ implies that the exponential map is a global diffeomorphism \cite[Theorem 1.2.1]{corwin1990representations}. This property allows us to define a global coordinate system using the exponential map as follows. The group \(\G\) admits a global parametrization $\Phi:\Ho \times \Ag \to \G$. A pair $(x,\theta)\in \Ho \times \Ag $ defines a point $g\in\G$ via
$$ g = \Phi(x,\theta) := \exp(\theta_{0} E_{0}^{\Laa}) * \prod_{k=2}^{s} \prod_{j=1}^{n_k} \exp(\theta_{j,k} E_{j,k}^{\Laa}) * \prod_{i=1}^{n} \exp(x_i E_i),  $$
where $*$ is multiplication on $\G$ and $\exp:\Lag \to \G$ is the exponential map. The coordinates $(x,\theta)$ are called exponential coordinates of second type.

The exponential coordinates of second type yield the following lemma.

\begin{lemma}[Lemma 3.2, \cite{BravoDoddoli2024}]\label{lem:exp-coor-sec}
Let $\G = \Ho \ltimes \Ag$ be a \ma Carnot group satisfying the conditions $dim \;\Lah = n$ and  $dim \; \Lav = 1$. Then the left-invariant vector fields generating the distribution $\Di$ are given by
\begin{equation*}
\begin{split}
X_i(g) & := (L_g)_* E_i  = \frac{\partial}{\partial x_i}\;\; 1 \leq i \leq n, \\
Y(g) & := (L_g)_* E_{0}^{\Laa} = \frac{\partial}{\partial \theta_{0}} + \sum_{k=2}^s\sum_{j=1}^{n_k}  P_{j,k}(x) \frac{\partial}{\partial \theta_{j,k}} ,
\end{split}
\end{equation*}
where $P_{j,k}:\Ho \to \R$ is a homogeneous polynomial of degree $k-1$.
\end{lemma}
 In these coordinates, the projection $\pi:\G \to \V$  has the form $\pi(x,\theta) = (x,\theta_0)$. In this paper, we consider the canonical Euclidean inner product on $\V_1$. Under this construction, the sub-Riemannian metric is given by $ds^2_{\G} = (dx_1^2 + \dots +dx_n^2 + d\theta_0^2)|_{\D}$. Thus, the frame $\{X_1,\dots,X_n,Y\}$ is orthonormal. 

Under the assumptions of this section, the reduced Hamiltonian from Eq. \eqref{eq:red-Ham} is a mechanical system with $n$ degrees of freedom whose potential is $\frac{1}{2}F^2_{\mu}(x)$, where $F_{\mu}(x)$ is determined by the linear bijection
\begin{equation}\label{eq:F-pol}
\mu= (a_0,a_{1,2},\dots,a_{n_s,s}) \to F_\mu(x) =  a_{0}   + \sum_{k=2}^s\sum_{j=1}^{n_k} a_{j,k-1} P_{j,k}(x).
\end{equation}

The following result implies that we need to consider only one asymptotic direction when classifying metric lines.
\begin{lemma}
The reflection $R_Y(g) = \Phi(x,-\theta)$ defines an isometry of $\G$.  
\end{lemma}
\begin{proof}
  The linear map $R_{\Laa}:\Lag \to \Lag$, defined by $E_i \to E_i$, $E_0^{\Laa} \to -E_0^{\Laa}$ and $E_{j,k}^{\Laa} \to -E_{j,k}^{\Laa}$  is a Lie automorphism. Thus, since $\G$ is simply connected, $R_{\Laa}$ induces an automorphism  $R_Y:\G \to \G$ \cite[Theorem 5.6]{hall2003lie}. Then $R_Y$ is a \sR isometry. 
\end{proof}
It follows that if $\gamma$ is a geodesic whose momentum $\mu$ determines the polynomial $F_{\mu}$, then $R_Y\circ\gamma$ is a geodesic whose momentum $-\mu$ determines the polynomial $-F_{\mu}$. Consequently, it suffices to study singular geodesics satisfying $F_{\mu}(x(t)) \to 1$ as $t \to \pm \infty$.

\subsection{The Magnetic Space}\label{sec:mag-spa}

Consider a metabelian Carnot group $\G = \Ho \ltimes \Ag$ satisfying the assumptions of Section \ref{sub-sub-sec:exp-coor}. In particular, the first layer $\V_1$ admits the decomposition $ \V_1 = \Lah \oplus\Lav$ with $\dim \Lav = 1$. In this setting, we select a geodesic in $\G$ as a candidate for a metric line. For simplicity, we consider a homoclinic geodesic $\gamma_h$ in $\G$ (the subsequent construction also applies to a direct-type geodesic). Assume that $\gamma_h$ has momentum $\mu$, and consider the polynomial $F_\mu$  determined by $\mu$ according to Remark \ref{rem:bij-a-vect-valued}. To simplify notation, denote $F_{\mu}$ by $F$. The following construction defines a sub-Riemannian magnetic space whose geometry depends on $F$. An explanation of the name \enquote{magnetic space} is given in \cite[Section 4.1]{bravo2022geodesics}.

Consider $\MS := \Ho \times \R^2$, equipped with global coordinates $(x,y,z)$. The distribution $\Di_{F}$ is defined by the Pfaffian equation $ dz - F(x) dy = 0$, and $\Di_{F}$ is endowed with the sub-Riemannian metric 
$$ds^2_{\MS} = ( \sum_{i=1}^n dx_i^2 + dy^2)\Big|_{\Di_F}.$$
The \sR manifold $\MS$ generalizes the magnetic spaces constructed in \cite{bravo2022geodesics,bravo2022metric}.

A sub-Riemannian submersion   $\pi_{F}:\G \to \MS$ is constructed to factor the projection $\pi:\G \to \R^{n+1}$, i.e.,  $\pi = pr \circ \pi_{F}$, where $pr: \MS \to \R^{n+1}$ is also a projection. By Lemma \ref{lem:sub-submersion}, it suffices to show that $c_h := \pi_F\circ\gamma_h$ is a metric line in $ \MS$ to conclude that $\gamma_h$ is a metric line. Thus, the goal of the sequence method is to prove that $c_h$ is a metric line.

If $F$ is given by Eq. \eqref{eq:F-pol}, then the projections $\pi_{F}$ and $pr$ are given by
\begin{equation*}\label{eq:proojection-F}
\begin{split}
\pi_{F}(x,\theta) = (x,\theta_0,a_0\theta_{0}   + \sum_{k=2}^s\sum_{j=1}^{n_k} a_{j,k-1} \theta_{j,k} ) = (x,y,z), \;\;\text{and}\;\; pr(x,y,z) := (x,y).  
\end{split}
\end{equation*}
 It follows that $\pi_{F}$ maps  the frame $\{X_1,\dots,X_n,Y \}$, defined in Lemma \ref{lem:exp-coor-sec}, into the frame $\{ \tilde{X}_1,\dots, \tilde{X}_n ,\tilde{Y}\}$, i.e.,
\begin{equation*}
\begin{split}
\tilde{X}_i & := \frac{\partial}{\partial x_i} = (\pi_{F})_* X_i\;\;\text{for all}\; i=1,\dots,n, \\
 \tilde{Y} & := \frac{\partial}{\partial y} + F(x)\frac{\partial}{\partial z} = (\pi_{F})_* Y . \\
\end{split}
\end{equation*} 
Thus, $\{ \tilde{X}_1,\dots, \tilde{X}_n,\tilde{Y} \}$ is a global orthonormal frame for $\Di_{F}$.

\subsubsection{Geodesics in the Magnetic Space}\label{sub-sec:geo-mag}
The Hamiltonian function governing the geodesic flow in $\MS$ is 
\begin{equation}\label{eq:fund-mag-eq-F}
H_{F}(p_x,p_y,p_z,x,y,z) = \frac{1}{2} \sum_{i=1}^n p_{x_i}^2 + \frac{1}{2}(p_y + F(x) p_z)^2. 
\end{equation}
A curve $c:\R \to \MS$, $c(t) = (x(t),y(t),z(t))$, is a normal geodesic in $\MS$ parameterized by arc length if it is the projection of a trajectory of the geodesic flow satisfying the energy condition $H_{F} = \frac{1}{2}$. 

Since  $H_{F}$ is independent of the coordinates $y$ and $z$, the momenta $p_y$ and $p_z$ are constants of motion. This fact has a deeper consequence. Let $Iso(\MS)$ denote the isometry group of $\MS$. If $\varphi_{(y_0,z_0)}:\MS \to \MS$ represents the translation defined by
\begin{equation}\label{eq:iso-def}
    \varphi_{(y_0,z_0)}(x,y,z) = (x,y+y_0,z+z_0),\quad\text{then}\quad \varphi_{(y_0,z_0)}  \in Iso(\MS).
\end{equation}
To establish this result, note that $W$ is a Killing vector field on $\MS$ if and only if its momentum function $P_W$ Poisson commutes with the sub-Riemannian kinetic energy $H_{F}$; see \cite[p. 8]{montgomery2002tour}. The coordinate functions $p_y$ and $p_z$ are momentum functions corresponding to the vector fields 
\begin{equation}\label{eq:kil-vec-fil}
  W_1: = \frac{\partial}{\partial y}\qquad \text{and} \qquad W_2 := \frac{\partial}{\partial z}.   
\end{equation}

 Thus, the translations $\varphi_{(y_0,z_0)}$ define an action of $\R^2$ on $\MS$. The cotangent lift of the $\R^2$-action to $T^*\MS$ admits a momentum map sending $(p_x,p_y,p_z,x,y,z)\in T^*\MS$ to $\nu := (p_y,p_z)\in \R^2$, where we identify $(\R^2)^*$ with $\R^2$. In the same spirit as Remark 2.4, there is a linear bijection between a momentum $\nu \in \R^2$ and a polynomial $G_{\nu}^F \in Pen_{F}$, where $Pen_{F}$ is called the pencil of $F$ and is a two-dimensional linear space defined as
$$Pen_{F} := \{ G_{\nu}^F(x) = a + bF(x):  \nu:=(a,b) \in \R^2 \}.$$
Hence, specifying $\nu$ uniquely determines the function $G_{\nu}^F$.

 A curve $\gamma$ in $\G$ is the horizontal lift of $c$ via the projection $\pi_F$ if $\gamma$ satisfies the ODE given by
\begin{equation*}
\dot{\gamma}(t) =  \sum_{i=1}^n\dot{x}_i(t) X_i(\gamma(t)) + G_{\nu}^F(x(t)) Y(\gamma(t)).
\end{equation*}

The following proposition describes the geodesics in $\MS$, their lifts, and their relationship with the geodesics in $\G$. 
\begin{Prop}[Section 4.1, \cite{bravo2022geodesics}]
If $c $ in $\MS$ is a geodesic with momentum $\nu$, then the following statements hold:
\begin{itemize}
    \item The curve $(p(t),x(t)) \in  T^*\Ho$, where $p(t) = \dot{x}(t)$, is a solution to the Hamiltonian system defined by the function
    \begin{equation}\label{eq:red-hal-mag-spa}
H_{G_{\nu}^F}(p_x,x) := \frac{1}{2} \sum_{i=1}^{n} p_{x_i}^2  +  \frac{1}{2}(G_{\nu}^F(x))^2.
\end{equation}
    \item The coordinates $y(t)$ and $z(t)$ satisfy the equations
\begin{equation}\label{eq:ode-y-z}
\dot{y} = G_{\nu}^F(x(t)) \;\;\text{and}\;\; \dot{z} = G_{\nu}^F(x(t)) F(x(t)).
\end{equation}
\end{itemize}
Moreover, every geodesic in  $\MS$ is the projection under $\pi_{F}$ of a geodesic in $\G$ whose momentum determines a polynomial $G_{\nu}^F\in Pen_{F}$. Conversely, the horizontal lift of a geodesic in $\MS$ via $ \pi_F$ is a geodesic whose momentum determines a polynomial in $Pen_F$.
\end{Prop}

\begin{Remark}\label{rem:can-geo}
Let $\gamma_h:\R \to \G$ be the candidate geodesic whose momentum $\mu$ determines the polynomial $F$. The magnetic space $\MS$ has the property that the curve $c_h := \pi_F\circ\gamma_h$ is a geodesic in $\MS$. It follows that the geodesic $c_h$ has momentum $\bar{\nu} := (0,1)$; in other words $G^F_{\bar{\nu}}(x) = F(x)$.
\end{Remark}

We classify geodesics in $\MS$ according to their reduced dynamics, defined by the Hamiltonian $H_{G_{\nu}^F}$ in Eq. \eqref{eq:red-hal-mag-spa}, as in Section \ref{subsubsec:clas-geo-G}. 
In particular, $c_h$ is a homoclinic geodesic in $\MS$.

The following lemma characterizes the abnormal geodesics in  $\MS$.
\begin{lemma}\label{lemma:abn-geo-mag}
If $F$ is a polynomial of degree $2$ with a unique critical point, then $\MS$ has two types of abnormal extremals:

\textbf{(Vertical)}  An  abnormal extremal $c$ is called vertical if $\dot{c}(t)\in \operatorname{span}\{\widetilde Y\}_{c(t)}$ whenever $\dot{c}(t)$ exists, and $x(t) = x^*$ is a constant, where $x^*\in \Ho$ satisfies $dF|_{x^*} = 0$. A vertical abnormal extremal is a geodesic if and only if it is a vertical line-geodesic.


\textbf{(Horizontal isocontours)} An abnormal extremal $c$ is called a horizontal isocontour if $\dot{c}(t)\in \operatorname{span}\{\widetilde X_1,\ldots,\widetilde X_n\}_{c(t)}$ whenever $\dot{c}(t)$ exists, and $F(x(t)) = const$. A horizontal isocontour abnormal extremal $c$ is a geodesic if and only if $c$ is a horizontal line-geodesic. 


Therefore, $ \MS$ has no strictly abnormal geodesics.
\end{lemma}

The proof of Lemma \ref{lemma:abn-geo-mag} is provided in Appendix \ref{sub-AP:ab-geo}.

Lemma \ref{lemma:abn-geo-mag} introduces the terms horizontal and vertical to describe curves in $\MS$. Specifically, a curve $c$ in $\MS$ is called horizontal if $\dot{c}(t)\in \Lah$ for all time. Similarly, a curve $c$ in $\MS$ is called vertical if $\dot{c}(t) \in \Lav$ for all time.

\subsubsection{Conjugate Times}

This section shows that the boundary of the Hill set defines conjugate points of geodesics in $\MS$. 

Given a polynomial $G_{\nu}^F\in Pen_F$, the Hill set $\Omega_\nu$ of $G_{\nu}^F$ is defined analogously to Definition \ref{def:Hill-region}. Moreover, the boundary of the Hill set is split into two disjoint subsets $\partial \Omega_\nu = \partial \Omega_\nu^+ \cup \partial \Omega_\nu^-$, where
    $$\partial \Omega_\nu^+ = \{x \in \Omega_\nu : G_{\nu}^F(x) = 1 \}, \;\; \text{and} \;\; \partial \Omega_\nu^- = \{x \in \Omega_\nu : G_{\nu}^F(x) = -1 \}. $$ 

The following proposition describes the conjugate points in  $\MS$.
\begin{Prop}\label{prop:conj-point}
    Let $c$ denote a geodesic in $\MS$ with momentum  $\nu$, and assume that $x(t_0)\in \partial \Omega_\nu^+$ (or $\partial \Omega_\nu^-$). If there exists a time $t_{conj} > t_0$ such that $x(t_{conj})\in \partial \Omega_\nu^+$ (or $\partial \Omega_\nu^-$), then  $c(t_{conj})$ is a conjugate point through $c(t_0)$, and $t_{conj}$ is a conjugate time. 
\end{Prop}

\begin{proof}
    We use the classical result that restricting a Killing vector field to a geodesic yields a Jacobi field. 
    
    Consider a geodesic $c$ in $\MS$ with momentum $\nu$ such that $x(t_0)\in \partial \Omega_\nu^+$. The restriction of $F(x)$ to $\partial \Omega_\nu^+$ is constant. Indeed, if $x \in \partial \Omega_\nu^+$, then $G_{\nu}^F(x) = 1$, and it follows that $F(x) = \frac{1-a}{b}$ ($\partial\Omega_\nu^+ \neq \varnothing  $ if and only if $b\neq 0$). Consider the following Jacobi field defined in terms of the Killing vector fields from Eq. \eqref{eq:kil-vec-fil}:
    $$ W(t) = \big( W_1 + \frac{1-a}{b}W_2\big)\big|_{c(t)}.$$
    By Eq. \eqref{eq:ode-y-z}, $\dot{c}(t) -W(t) \not\equiv 0$. Moreover,  since $p_x(t) = 0$ and $G_{\nu}^F(x(t))=1$ when $x(t) \in \partial\Omega_\nu^+$, it follows that the Jacobi field $\dot{c}(t)- W(t)$ has the following property:
    $$ \dot{c}(t) - W(t) = 0\quad \text{when} \; t=t_0\;\text{or}\;\;t_{conj}. $$
    Thus, $c(t_{conj})$ is a conjugate point through the point $c(t_0)$. 
\end{proof}

\begin{Remark}
    It is well known that a geodesic ceases to be minimizing after passing through a conjugate point.
Consequently, a geodesic $c$ in $\MS$ cannot be minimizing after reaching $\partial\Omega_\nu^+$ twice (the same is true for $\partial\Omega_\nu^-$).
\end{Remark}

\subsubsection{Cost Map in Magnetic Space}
The cost map is an essential tool for the sequence method. In \cite{bravo2022geodesics,bravo2022metric}, it was defined and used to establish results on metric lines in the jet spaces. We now generalize this definition to a broader class of magnetic spaces. The following two sets play a fundamental role in the paper.
\begin{equation}
\begin{split}
    \mathcal{C} & := \{c:\R \to \MS : \; c \; \text{is an arc-length-parameterized geodesic} \}, \\
    \mathcal{I} & := \{ [t_0,t_1]  \subset \R : \; t_0<t_1 \}. \\
\end{split}
\end{equation}

 Denote by $\Delta: \mathcal{C} \times \mathcal{I} \to [0,\infty) \times \R^2$ the map given by
\begin{equation}\label{eq:def-Delta}
\begin{split}
\Delta(c,\mathcal{T}) & := (\Delta t (c,\mathcal{T}),\Delta y (c,\mathcal{T}), \Delta z (c,\mathcal{T})) , \\
\end{split}
\end{equation}
where
$$  \Delta t (c,\mathcal{T}) := t_1 - t_0,\;\; \Delta y (c,\mathcal{T}) := y(t_1) - y(t_0),\;\; \Delta z (c,\mathcal{T}) :=z(t_1) - z(t_0).$$
The map $\Delta(c,\mathcal{T})$ encodes the changes in time and in the coordinates $y$ and $z$ as the geodesic $c$ travels over the time interval $ \mathcal{T}$. In addition, denote by $\operatorname{Cost}: \mathcal{C} \times \mathcal{I} \to [0,\infty) \times \R$ the map given by
\begin{equation}\label{eq:def-cost-f}
\begin{split}
\operatorname{Cost}(c,\mathcal{T})&:= ( \operatorname{Cost}_t(c,\mathcal{T}) , \operatorname{Cost}_y(c,\mathcal{T})),\\       
\end{split}
\end{equation} 
where
\begin{equation*}
\begin{split}
\operatorname{Cost}_t(c,\mathcal{T}) := & \Delta t(c,\mathcal{T}) - \Delta y(c,\mathcal{T}), \;\;\operatorname{Cost}_y(c,\mathcal{T}) :=  \Delta y(c,\mathcal{T}) - \Delta z(c,\mathcal{T}). \\
\end{split}
\end{equation*} 
The map $\operatorname{Cost}(c,\mathcal{T})$ is called the cost map.

The map $\operatorname{Cost}_t(c,\mathcal{T})$ is interpreted as the cost required for the geodesic $c$ to travel in the positive $ y$-direction. To clarify this interpretation, we present the following lemma.
\begin{lemma}[Lemma 21, \cite{bravo2022metric}]\label{lem:cost-fun-int}
Let $c$ and $\tilde{c}$ be two geodesics in $\MS$ joining a point $p$ to a point $q$ over the time intervals $\mathcal{T}$ and $\widetilde{\mathcal{T}}$, respectively. If $\operatorname{Cost}_t(c,\mathcal{T}) < \operatorname{Cost}_t(\tilde{c},\widetilde{\mathcal{T}})$, then the arc length of $c|_\mathcal{T}$ is less than that of $\tilde{c}|_{\widetilde{\mathcal{T}}}$.
\end{lemma}

 Let $c$ be a geodesic with momentum $\nu$ over a time interval $\mathcal{T} = [t_0,t_1]$. The objective is to express the maps $\Delta(c,\mathcal{T})$ and $\operatorname{Cost}(c,\mathcal{T})$ in terms of $G_{\nu}^F$. To this end, denote by $s(\T)$ the arc length of $x|_\T$, that is,
$$ s(\T)  := \int_{t_0}^{t_1} \sqrt{ p_{x_1}^2(t)+ \dots + p_{x_n}^2(t)} dt = \int_{t_0}^{t_1} \sqrt{ 1-(G^F_{\nu}(x(t)))^2 } dt. $$
If $c$ is not a vertical line-geodesic, it follows that $s(t) := s([t_0,t])$ is strictly increasing, since \(\|\dot x(t)\|>0\) except possibly at isolated times. Hence its inverse $t(s)$ is well defined, and $x$ admits a parameterization in terms of $s$. 

\begin{Prop}\label{prop:mag-geo-Delta-C}
Consider the map $\Delta:\mathcal{C}\times\mathcal{I} \to [0,\infty) \times\R^2$ given by Eq. \eqref{eq:def-Delta}. The map $\Delta(c,\mathcal{T})$ can be rewritten in terms of $G_{\nu}^F(x)$ and the arc-length interval $\mathcal{S} := [0,s(\mathcal{T})]$ as follows:

\begin{equation*}
\begin{split}
\Delta(c,\mathcal{T}) =  ( \int_\mathcal{S} \frac{ds}{\sqrt{1-(G_{\nu}^F(s))^2}}, \int_\mathcal{S}  \frac{G_{\nu}^F(s) ds}{\sqrt{1-(G_{\nu}^F(s))^2}}, \int_\mathcal{S}  \frac{G_{\nu}^F(s) F(s) ds}{\sqrt{1-(G_{\nu}^F(s))^2}} ).
\end{split}
\end{equation*}
To simplify notation, $F(x(s))$ and $G_{\nu}^F(x(s))$ are denoted by $F(s)$ and $G_{\nu}^F(s)$, respectively.

Similarly, consider the map $\operatorname{Cost}: \mathcal{C}\times\mathcal{I} \to [0,\infty) \times\R$ from Eq. \eqref{eq:def-cost-f}. The map $\operatorname{Cost}(c,\mathcal{T})$ can be rewritten as follows:
\begin{equation*}
\begin{split}
\operatorname{Cost}(c,\mathcal{T})  &  =  \Big( \int_\mathcal{S} \frac{1- G_{\nu}^F(s)}{\sqrt{1-(G_{\nu}^F(s))^2}} ds,  \int_\mathcal{S} \frac{ G_{\nu}^F(s)\big(1-F(s)\big)}{\sqrt{1-(G_{\nu}^F(s))^2}} ds\Big).\\
\end{split}
\end{equation*}
\end{Prop}

\begin{proof}
Make the change of variables  $t = t(s)$. Then $ds = \sqrt{ 1-(G^F_{\nu}(s))^2 } dt$, and this expression yields the formula for $\Delta t(c,\mathcal{T})$. For $\Delta y(c,\mathcal{T})$ and $\Delta z(c,\mathcal{T})$, use the same change of variables together with the expressions $\dot{y} = G_{\nu}^F(x(t))$ and $\dot{z} = G_{\nu}^F(x(t))F(x(t))$. For $\operatorname{Cost}_t(c,\mathcal{T})$ and $\operatorname{Cost}_y(c,\mathcal{T})$, consider the differences of the preceding expressions. 
\end{proof}

Consider the translation $\varphi_{(y_0,z_0)}\in Iso(\MS)$ defined in Eq. \eqref{eq:iso-def}. It follows that $\Delta(c,\T)$ is invariant under $\varphi_{(y_0,z_0)}$. Thus, $\operatorname{Cost}(c,\T)$ is also invariant. Specifically,  for all $(c,\T) \in \mathcal{C}\times \mathcal{I}$, the following identities hold:
$$\Delta(c,\T)= \Delta(\varphi_{(y_0,z_0)}\circ c,\T)\;\;\text{and} \;\; \operatorname{Cost}(c,\T) = \operatorname{Cost}(\varphi_{(y_0,z_0)}\circ c,\T).$$
Therefore, $\Delta(c,\T)$ and $\operatorname{Cost}(c,\T)$ remain invariant under the action of a nontrivial subgroup of $Iso(\MS)$. Since isometries send metric lines to metric lines, the classification of metric lines may be reduced to the study of geodesics modulo isometry.

Let $Iso^*(\MS)\subset Iso(\MS)$ be the subgroup of isometries that preserve $\Delta(c,\T)$. Two geodesics $c_1,c_2 \in \mathcal{C}$ are equivalent if and only if there exists an isometry $\varphi\in Iso^*(\MS)$ such that $c_1 = \varphi \circ c_2$. We denote the equivalence class of $c$ by $[c]$. 

\begin{Remark} 
   All representatives of the class $[c]$ have the same momentum $\nu$. The equivalence relation on the space of geodesics induces an equivalence relation on the space of solutions $x$ of the reduced system. Consequently,  each class $[c]$ corresponds to a unique pair $(G_{\nu}^F,[x])$. Moreover, if $c$ is a homoclinic geodesic, then every representative of $[c]$ is a homoclinic geodesic, and all such representatives determine the same homoclinic orbit class. The same statement holds, mutatis mutandis, for heteroclinic geodesics.
\end{Remark}


We now define the set of homoclinic orbit classes asymptotic to a fixed point, as well as the set of heteroclinic orbit classes asymptotic to a pair of fixed points.  Consider two points $x_0,x_1 \in \Ho$. 
\begin{itemize}
    \item Let $Homc(x_0)$ denote the set of equivalence classes $[c]$ whose representatives have reduced dynamics exhibiting a homoclinic orbit asymptotic to $x_0$. The set $Homc(x_0)$ is the disjoint union of two subsets, namely 
$$Homc^{\pm}(x_0) := \{ (G_{\nu}^F,[x]) : G_{\nu}^F(x_0) = \pm 1 \}.$$
If the magnetic space $\MS$ is defined by a polynomial $F$ with a unique critical point, we omit  $x_0$ from the notation.

\item Let $Hetc(x_0,x_1)$ denote the set of equivalence classes $[c]$ whose representatives have reduced dynamics exhibiting a heteroclinic orbit asymptotic to $x_0$ as $t\to-\infty$ and to $x_1$ as $t \to \infty$. The set $Hetc(x_0,x_1)$ is the union of two disjoint subsets, namely $$Hetc^{\pm}(x_0,x_1) := \{ (G_{\nu}^F,[x]) : G_{\nu}^F(x_0) = \pm 1\;\;\text{and}\;\; G_{\nu}^F(x_1) = \pm 1 \}.$$
\end{itemize}

A class $[c]$ is determined by a polynomial $G_{\nu}^F$ and an orbit class $[x]$. Section \ref{sub-sec:proof-theo-main-1} uses this fact to characterize the space $Homc^{\pm}(x_0)$.

The period map  $\Theta: Homc^+(x_0) \to [0,\infty)\times\R$ is defined by  
\begin{equation*}
\begin{split}
\Theta([c]) = \big(\Theta_1([c]),\Theta_2([c])\big) & :=  \Big(\int_{\mathcal{S}_{\infty}}  \frac{1- G_{\nu}^F(s)}{\sqrt{1-(G_{\nu}^F(s))^2}} ds , \int_{\mathcal{S}_{\infty}} \frac{ G_{\nu}^F(s)\big(1-F(s)\big)}{\sqrt{1-(G_{\nu}^F(s))^2}} ds\Big),
\end{split}
\end{equation*} 
where $\mathcal{S}_{\infty}$ is the interval parametrizing the homoclinic orbit by arc-length, obtained as the limit of the arc length intervals \(S_n:=[0,s(\T_n)]\) as $\T_n$ exhausts $\R$.

If $[c]\in Homc^-(x_0)$, then $\Theta_1([c])$ is not well defined. The following lemma is essential for proving the sequence method.

\begin{lemma}\label{lem:uniform-bound-c-s}
Consider the candidate geodesic $\gamma_h$ whose momentum determines the polynomial $F$, and $c_h := \pi_F\circ \gamma_h$. Then $\|\operatorname{Cost}(c_h,\mathcal{T}) \|_{\sup}\leq  \Theta_1([c_{h}])$, where $\|\cdot\|_{\sup}$ is the supremum norm.
\end{lemma}

\begin{proof}
By Remark \ref{rem:can-geo},  the geodesic $c_{h}$ has momentum $\bar{\nu} = (0,1)$, and $G_{\bar{\nu}}^F(x) = F(x)$. Thus, $|F(s)| \leq 1$ for all $s\in \mathcal{S}_{\infty}$. Using Proposition \ref{prop:mag-geo-Delta-C}, it follows that
\begin{equation*}
\begin{split}
| \operatorname{Cost}_y(c_h,\mathcal{T}) | & \leq \int_{\mathcal{S}} \frac{|F(s)| \big(1-F(s)\big)}{\sqrt{1-F^2(s)}}ds \\
& \leq  Cos_t(c_h,\mathcal{T}) \leq \int_{\mathcal{S}_{\infty}}\frac{1-F(s)}{\sqrt{1-F^2(s)}}ds = \Theta_1([c_h]). 
\end{split}
\end{equation*}
The last equality follows from the definition of \(\Theta_1\) and Remark \ref{rem:can-geo}.
\end{proof}

\subsubsection{Sequence of Geodesics}

This section presents essential results and definitions for the sequence method.

\begin{Prop}[Proposition 2.26, \cite{bravo2022metric}]\label{prp:seque-comp}
Let $K\subset \MS$ be compact, and let $\mathcal{T}$ be a compact time interval. Consider the following family of geodesics
\begin{equation*}
\begin{split}  Min(K,\mathcal{T}) := \big\{ & c \in \mathcal{C} :   c(\mathcal{T}) \subset K  \text{ and}\;c|_\mathcal{T} \text{ is minimizing}  \big\}. 
\end{split}
\end{equation*}
 Then $Min(K,\mathcal{T})$ is sequentially compact with respect to the uniform topology. 
\end{Prop}


\begin{lemma}\label{lem:iso-metr}
Let $c_1,c_2\in\mathcal{C}$ be two geodesics such that $c_1\in Min(K,\mathcal{T})$. If $\varphi:\MS\to\MS$ is an isometry such that $c_2(\mathcal{T}') \subset \varphi(c_1(\mathcal{T}))$, then $c_2 \in Min(\varphi(K),\mathcal{T}')$. 
\end{lemma}

\begin{defi}\label{def:geo-comp-def}
Consider $r>0$, a point $x_0 \in \Ho$, and a constant $C>0$. Define the family $\\operatorname{Com}(r,x_0,C) \subset \mathcal{C} \times \mathcal{I}$ as the set of pairs $(c,\mathcal{T})$ satisfying the following conditions:
\begin{enumerate}
\item $c$ is a normal geodesic and $c|_{\mathcal{T}}$ is minimizing.

\item $\|\operatorname{Cost}(c,\mathcal{T})\|_{sup} \leq C$.

\item $x(\partial \mathcal{T})\subset B_{\Ho}(r,x_0)$, where $B_{\Ho}(r,x_0) \subset \Ho$ is the closed ball, and $\partial \mathcal{T}$ is the boundary of $\mathcal{T}$.
\end{enumerate}

The region $B_{\Ho}(r,x_0) \times \R^2$ is called geodesically compact if there exists $C>0$ such that, for every sequence $\{(c_n,\mathcal{T}_n)\}_{n\in \mathbb{N}} \subset \operatorname{Com}(r,x_0,C)$ such that $\mathcal{T}_n$ exhausts $\R$ as $n \to \infty$,
there exists a constant $S_C$ such that  $s(\mathcal{T}_n) < S_C$ for all $n$. 

A magnetic space is geodesically compact if $B_{\Ho}(r,x_0) \times \R^2$ is geodesically compact for every $r>0$. 
\end{defi}

Since each $c_n$ in a sequence $\{(c_n,\mathcal{T}_n)\}_{n\in \mathbb{N}} \subset \operatorname{Com}(r,x_0,C)$ is normal, we can associate to each $c_n$ a polynomial $G_{\nu_n}^F$. We say that \((c_n,\mathcal T_n)\to(c,\mathcal T)\) if \(G_{\nu_n}^F\to G_{\nu}^F\) and \(x_n\to x\) uniformly on compact subsets of their respective domains and \(\mathcal T_n\to\mathcal T\).

As a consequence, if $B_{\Ho}(r,x_0) \times \R^2$ is geodesically compact, then there exists a compact subset $K_{C} \subset \Ho$ such that $x_n(\mathcal{T}_n) \subset K_{C}$ for all $n \in \mathbb{N}$, where $x_n$ denotes the reduced trajectory of $c_n$ and  $\{(c_n,\mathcal{T}_n)\}_{n\in \mathbb{N}} \subset \operatorname{Com}(r,x_0,C)$.


\begin{defi}\label{def:str-nor-sec}
A sequence of geodesics $\{c_n\}_{n\in \mathbb{N}}$ is called precisely normal if each $c_n$ is a normal geodesic and every convergent subsequence converges to a normal geodesic.  
\end{defi}


\begin{lemma}\label{lem:com-converge}
Consider a magnetic space $\MS$ with a geodesically compact region $B_{\Ho}(r,x_0) \times \R^2$.  If $\{(c_n,\T_n)\}_{n\in \mathbb{N}}\subset \operatorname{Com}(r,x_0,C)$ converges to $(c,\R)$, where $c$ is a geodesic with momentum $\nu$ that is not a line-geodesic, then 
    $$\lim_{t \to \pm \infty}G^F_{\nu}(x(t)) = 1, \;\text{and}\;\lim_{t \to \pm \infty}F(x(t)) = 1.$$
\end{lemma}

\begin{proof}
Set $u(t):=1-G_\nu^F(x(t))$. Since $|G_\nu^F(x(t))|\leq 1$, we have $u(t)\geq 0$. For every $R>0$, we have $[-R,R]\subset T_n$ for all sufficiently large $n$. Hence,
\[
\int_{-R}^R \left(1-G_{\nu_n}^F(x_n(t))\right)dt
\leq \operatorname{Cost}t(c_n,T_n)\leq C.
\]
Since \(G{\nu_n}^F(x_n(t))\to G_\nu^F(x(t))\) uniformly on compact intervals, it follows that
\[
\int_{-R}^R u(t)dt\leq C.
\]
Letting \(R\to\infty\), we obtain
\[
\int_{\mathbb R}u(t)dt\leq C.
\]

By geodesic compactness, there exists a compact set \(K_C\subset H\) such that \(x_n(T_n)\subset K_C\). Since \(x_n\to x\) uniformly on compact intervals and \(T_n\to\mathbb R\), we have \(x(\mathbb R)\subset K_C\). Since \(G_\nu^F\) is smooth, there exists \(C_1>0\) such that
\[
\sup_{x\in K_C}|dG_\nu^F|x|\leq C_1.
\]
Therefore,
\[
|u'(t)|
=\left|dG_\nu^F|_{x(t)}(\dot x(t))\right|
\leq C_1,
\]
where we have used \(\|\dot x(t)\|\leq1\).

Suppose that \(u(t)\not\to 0\) as \(t\to+\infty\). Then there exist \(\varepsilon>0\) and a sequence \(t_j\to \infty\) such that \(u(t_j)\geq\varepsilon\). The bound on \(u'\) implies that
\[
u(t)\geq\frac{\varepsilon}{2}
\qquad\text{if}\qquad
|t-t_j|\leq\frac{\varepsilon}{2C_1}.
\]
By passing to a subsequence, we may assume these intervals are pairwise disjoint. It follows that
\[
\int_{\mathbb R}u(t)dt=\infty,
\]
which is a contradiction. Thus \(u(t)\to 0\) as \(t\to\infty\). The same argument shows that $u(t) \to 0$ as \(t\to-\infty\). Therefore,
\[
\lim_{t\to\pm\infty}G_\nu^F(x(t))=1.
\]

It remains to prove that $F(x(t))\to 1$ as $t\to \pm \infty$. Write
\[
G_{\nu_n}^F=a_n+b_nF
\qquad\text{and}\qquad
G_\nu^F=a+bF.
\]
Since $G_{\nu_n}^F\to G_\nu^F$, we have $a_n\to a$ and $b_n\to b$. Moreover, since \(c\) is not a line-geodesic, \(b\neq0\).


Since \(G_\nu^F=a+bF\) and \(G_\nu^F(x(t))\to 1\) as \(t\to\pm\infty\), it follows that
\[
F(x(t))\to L:=\frac{1-a}{b}.
\]
We claim that \(L=1\). Write \(G_{\nu_n}^F=a_n+b_nF\) and set
\[
u_n(t):=1-G_{\nu_n}^F(x_n(t)),\qquad
d_n:=\frac{a_n+b_n-1}{b_n}.
\]
Since \(b_n\to b\neq0\), \(d_n\to 1-L\). Moreover, \(0\le u_n\le2\) and
\[
\int_{T_n}u_n(t)dt=\operatorname{Cost}_t(c_n,T_n)\le C.
\]
Using $1-F(x_n(t)) = d_n+\frac{u_n(t)}{b_n}$, we have 
\begin{equation}\label{eq:cost-y-double-lemma}
   G_{\nu_n}^F(x_n(t))(1-F(x_n(t))) = d_n +(\frac{1}{b_n}-d_n)u_n(t) - \frac{u^2_n(t)}{b_n}. 
\end{equation}
It follows that
\[
\operatorname{Cost}_y(c_n,T_n)
=d_n|T_n|
+\left(\frac1{b_n}-d_n\right)\int_{T_n}u_n(t)dt
-\frac1{b_n}\int_{T_n}u_n(t)^2dt.
\]
Since \(u_n^2\le2u_n\), the last two terms are uniformly bounded. Thus
\[
\operatorname{Cost}_y(c_n,T_n)=d_n|T_n|+O(1).
\]
If \(L\neq 1\), then \(d_n\to1-L\neq0\), and since \(|T_n|\to\infty\), we obtain
\[
|\operatorname{Cost}_y(c_n,T_n)|\to\infty,
\]
contradicting \(\|\operatorname{Cost}(c_n,T_n)\|_{\sup}\le C\). Therefore \(L=1\).
\end{proof}

\begin{Remark}
A direct computation shows that if $\gamma$ is a geodesic in $\G$ and $c = \pi_F\circ\gamma$ satisfies $\Theta_1([c]) < \infty$, then $\gamma$ satisfies the condition stated in Proposition \ref{prop:met-char}. Therefore, focusing on geodesics with this property is justified.
\end{Remark}

\section{Metric Lines and the Sequence Method}\label{sec:sequ-meth}

This section proves Proposition \ref{prop:met-char} and presents the sequence method, summarized in Theorem \ref{the:metrtic-lines-method}.

\subsection{Characterization of Metric Lines}
Before proving Proposition \ref{prop:met-char}, we introduce Carnot dilations and curve blowdowns. These definitions are fundamental for the argument that follows.

Carnot groups admit dilations \cite{montgomery2002tour,agrachev2019comprehensive}. The dilations form a one-parameter group of automorphisms of $\G$, denoted by $\delta_u: \G \to \G$, where $u\in \R \setminus \{0\}$. These dilations are compatible with the metric; that is, $dist_{\G}(\delta_u g_1,\delta_u g_2) = |u| dist_{\G}( g_1, g_2)$. If $u \neq 0$ and $\gamma$ is a geodesic parameterized by arc length, so is $\gamma_u$, where 
$$  \gamma_u(t) := \delta_{\frac{1}{u}} \gamma(ut) . $$
 The following lemma says that the dilation of a metric line is a metric line.
\begin{lemma}[Lemma 7, \cite{bravo2022metric}]\label{lem:carnor-dil}
If $\gamma$ is a metric line in a Carnot group $\mathbb{G}$, then $\gamma_u$ is a metric line in $\mathbb{G}$.
\end{lemma}

E. Hakavuori and E. Le Donne developed a method to prove that a geodesic is not a metric line. This idea relies on the concept of a blowdown: let $\G$ be a Carnot group, and $\gamma$ be a rectifiable curve in $\G$. A curve $\bar{\gamma}$ is called a blowdown of $\gamma$ if $\bar{\gamma}(t) = \lim_{n \to \infty} \gamma_{u_n}(t)$, where $u_n\to \infty$ and the convergence is uniform on compact intervals.

E. Hakavuori and E. Le Donne proved the following proposition.

\begin{Prop}[Corollary 1.6, \cite{hakavuori2023blowups}] \label{prp:blow-down}
Let $\G$ be a Carnot group. Let $\gamma:\R \to \G$ be a metric line. Then there exists a sequence $\{u_n\}_{n\in \mathbb{N}}$, with $u_n \to \infty$, such that the blowdown $\gamma_{u_n}$ converges uniformly on compact sets to an arc-length-parameterized line-geodesic $\bar{\gamma}$.
\end{Prop}

We now prove Proposition \ref{prop:met-char}.

\begin{proof}[Proof of Proposition \ref{prop:met-char}]
Assume $\gamma$ is a metric line. Then Proposition \ref{prp:blow-down} implies there exists a sequence \(u_n\to\infty\) such that \(\gamma_{u_n}\) converges uniformly on compact sets to a line-geodesic $\bar\gamma$, given by \(\bar\gamma(t)=g\exp(tv)\), where \(v\in V_1\) is a unit vector.  If $\mathcal{T} = [-1,1]$ and  $\mathcal{T}_n := [-u_n,u_n]$, then
\begin{equation*}
\begin{split}
\frac{1}{2}\int_{\mathcal{T}} \langle v,(L_{\gamma^{-1}_{u_n}(t)})_*\dot{\gamma}_{u_n}(t)\rangle dt & = \frac{1}{2}\int_{\mathcal{T}} \langle v,(L_{(\delta_{\frac{1}{u_n}}\gamma(u_n t))^{-1}})_*\frac{d}{dt} (\delta_{\frac{1}{u_n}}\gamma(u_n t))\rangle dt \\
          & = \frac{1}{2u_n} \int_{\mathcal{T}_n} \langle v,(L_{\gamma^{-1}(\tilde{t})})_*(\dot{\gamma}(\tilde{t}))\rangle d\tilde{t} \\
\end{split}
\end{equation*}
The first equality follows from the definition of $\gamma_{u_n}(t)$, and the second follows from the change of variables $u_n t = \tilde{t}$. Taking the limit as \(n\to\infty\) and using the uniform convergence of \(\gamma_{u_n}\) to \(\bar\gamma\) on $\mathcal{T}$, we obtain     
\begin{equation*}
\begin{split}
 1 = \frac{1}{2}\int_{\mathcal{T}} \langle v,v\rangle dt & =  \lim_{n \to \infty} \frac{1}{2u_n} \int_{\mathcal{T}_n} \langle v,(L_{\gamma^{-1}(\tilde{t})})_*(\dot{\gamma}(\tilde{t}))\rangle d\tilde{t}\\
 & \leq \limsup_{t \to\infty}  \frac{1}{2t} \int_{-t}^t \langle v,(L_{\gamma^{-1}(\tilde{t})})_*(\dot{\gamma}(\tilde{t}))\rangle d\tilde{t} \leq 1. \\
\end{split}
\end{equation*}
\end{proof}

The following lemma shows that the condition in Conjecture \ref{conj} implies the condition stated in Proposition \ref{prop:met-char}. 

\begin{lemma}\label{lem:cond-metl-lin-2}
Let $\G$ be a Carnot group. Then the following statements hold:
\begin{itemize}
\item Conjecture \ref{conj} implies the condition in Proposition \ref{prop:met-char}.

\item If $\gamma$ is a geodesic such that the following limits exist but are distinct,
$$\lim_{t \to \infty}(L_{\gamma^{-1}(t)})_*\dot{\gamma}(t) \;\; \text{and} \;\; \lim_{t \to -\infty}(L_{\gamma^{-1}(t)})_*\dot{\gamma}(t), $$
then $\gamma$ is not a metric line. Therefore, turn-back geodesics are not metric lines. 
\end{itemize}
\end{lemma}

\begin{proof}
Both statements are direct consequences of the following fact, which follows from L’Hôpital’s rule: if \(f\) is continuous and $\lim_{t \to \infty} f(t) = c$, then $\lim_{t \to \infty} \frac{1}{t} \int_0^t f(s)ds = c$. 
\end{proof}

\subsection{The Sequence Method} This section outlines the sequence method for homoclinic geodesics and emphasizes its relevance to the Engel-type group. Appendix \ref{ap:se-me} presents the sequence method for direct-type geodesics.

 Let $\G$ be a metabelian Carnot group equipped with a semidirect product structure. Consider a candidate homoclinic geodesic $\gamma_h:\R \to\G$ with momentum $\mu$. Let $F$ denote the polynomial determined by $\mu$, as introduced at the beginning of Section \ref{sec:mag-spa}. Define the geodesic 
$c_h := \pi_F\circ \gamma_h$, so that $$c_h(t)  = (x_h(t),y_h(t),z_h(t)).$$
Without loss of generality, assume that $[c_h]\in Homc^+(x_0)$, and 
\begin{equation}\label{eq:ini-cond-can}
c_h(0) = (x_h(0),0,0).    
\end{equation}
Indeed, by applying the reflection \(R_Y:G\to G\) and the translation \(\phi_{(y_0,z_0)}:M_F\to M_F\), we may assume that the geodesic \(c_h\) has these properties.

The goal is to show that, for every $T>0$, the geodesic \(c_h\) is length minimizing on $[-T,T]$.  The strategy is the following. For each $n >T$, let \(c_n\) be a geodesic minimizing on \([0,T_n]\) and joining \(c_h(-n)\) to \(c_h(n)\); see Figure \ref{fig:hom}. The sequence method constructs a compact subset  $K\subset \MS$ and a compact interval $\mathcal{T}$. It then extracts a subsequence \(\{c_{n_j}\}\subset\operatorname{Min}(K,\T)\) converging to a geodesic \(c_\infty\) and identifies an isometry \(\phi\in\operatorname{Iso}^*(M_F)\) such that $c_h([-T,T]) \subseteq \varphi( c_{\infty}(\mathcal{T}))$.  By Lemma \ref{lem:iso-metr}, $c_h$ is minimizing on  $[-T,T]$. Since $T$ is arbitrary, it follows that $c_h$ is a metric line.


\begin{figure}%
    \centering
    {{\includegraphics[width=5cm]{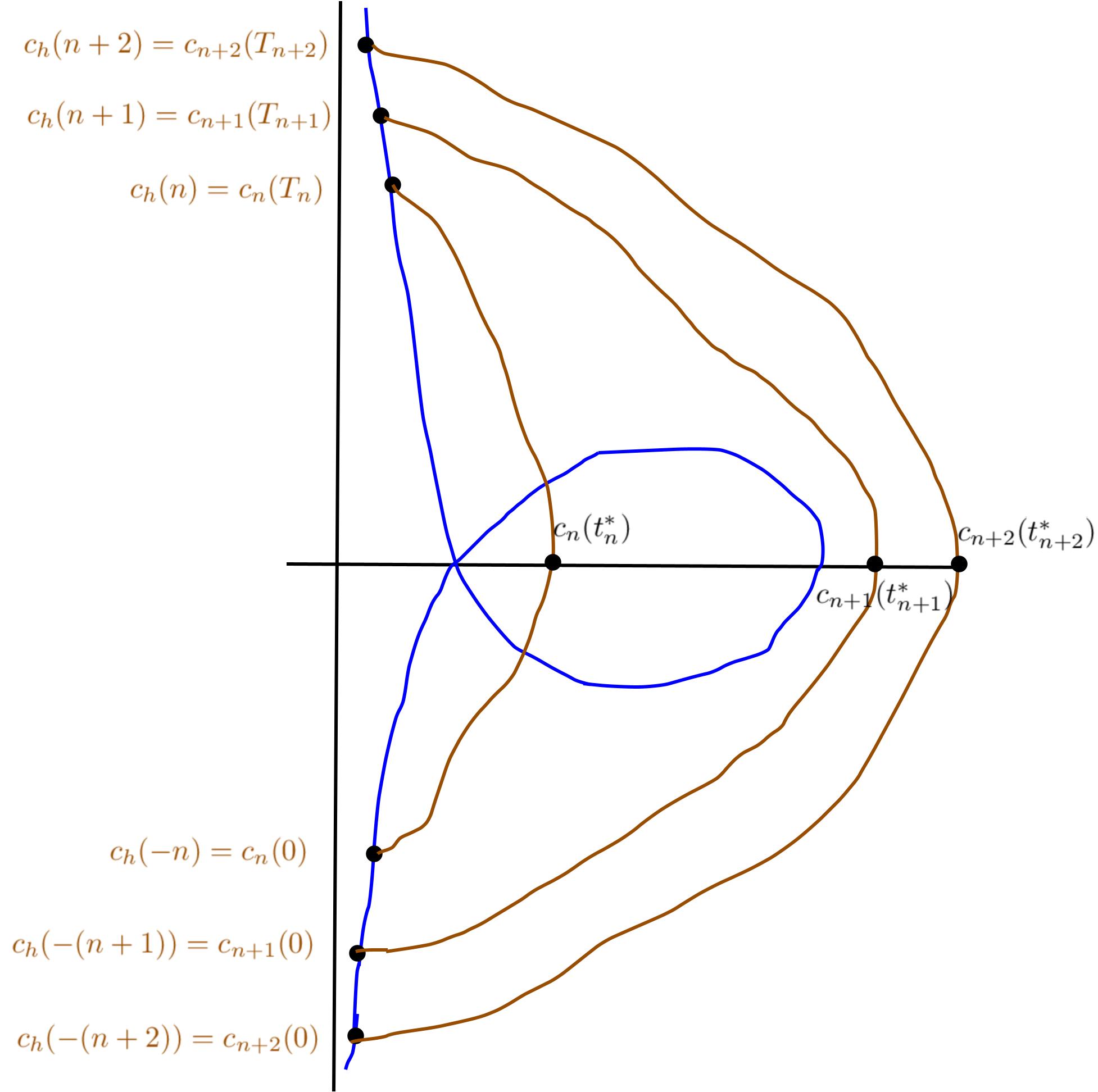} }}
    \caption{The images show the \((x,y)\)-plane and the projections under \(\operatorname{pr}\) of the homoclinic geodesic \(c_h\) and the sequence of minimizing geodesics \(\{c_n\}\).} 
    \label{fig:hom}
\end{figure}

\subsubsection{Setup}

Fix \(T>0\). For each $n\in \mathbb{N}$ with \(n>T\), let \(c_n\), given by \(c_n(t)=(x_n(t),y_n(t),z_n(t))\), be minimizing on \([0,T_n]\) and satisfying the following conditions:
\begin{equation}\label{eq:shot-con}
c_n(0) = c_h(-n),\;\;\; c_n(T_n) = c_h(n), \;\;\text{and}\;\; T_n \leq 2n.
\end{equation}
The equalities in Eq. \eqref{eq:shot-con} are called the endpoint conditions, and the inequality is the length condition. Since the endpoint conditions hold for all \(n\), the sequence \(\{c_n\}\) satisfies the following asymptotic conditions:
\begin{equation}\label{eq:asymp-cond-h}
\lim_{n \to \infty} c_n(0) = (x_0,-\infty,-\infty) \;\;\; \text{and} \;\;\; \lim_{n \to \infty} c_n(T_n) = (x_0,\infty,\infty).
\end{equation}
In addition, the endpoint conditions also imply
$$ \Delta y(c_h,[-n,n]) = \Delta y(c_n,[0,T_n]), \;\;\;\text{and}\;\; \Delta z(c_h,[-n,n]) = \Delta z(c_n,[0,T_n]), $$
for all $n$. Thus, the endpoint and length conditions imply
\begin{equation}\label{eq:cots-fuc-cond}
\begin{split}
\operatorname{Cost}_t(c_h,[-n,n]) & \geq \operatorname{Cost}_t(c_n,[0,T_n]), \\
\operatorname{Cost}_y(c_h,[-n,n]) & = \operatorname{Cost}_y(c_n,[0,T_n]), \\
\end{split}
\end{equation} 
for all $n$. Eq. \eqref{eq:cots-fuc-cond} yields the asymptotic period condition:
\begin{equation}\label{eq:asymp-cond-theta}
\lim_{n \to \infty} \operatorname{Cost}_y(c_n,[0,T_n]) = \Theta_2([c_h]).
\end{equation}
\begin{Remark}
Eq. \eqref{eq:asymp-cond-theta} does not state that $\{c_n\}$ converges to $c_h$. It only states that the sequence $\{\operatorname{Cost}_y(c_n,[0,T_n] )\}_{n\in \mathbb{N}}$ converges to $\Theta_2([c_h])$.    
\end{Remark}

The following theorem formalizes the sequence method. 
\begin{Theorem}[Sequence method for a homoclinic geodesic]\label{the:metrtic-lines-method}
Consider a \ma Carnot group $\G$ admitting a semidirect product structure with $\dim \Vs = 1$. Let $\gamma_h:\R\to \G$ be a homoclinic geodesic whose momentum $\mu$ determines the polynomial $F$. Consider the magnetic space $\MS$ and the geodesic $c_h := \pi_F\circ\gamma_h$, with $[c_h]\in Homc^+(x_0)$. Suppose that the following conditions hold:
\begin{enumerate}

\item \label{cond-1} The value $\Theta_1([c_h])$ is finite.

\item \label{cond-2} There exists a ball $B_{\Ho}(r,x_0)$ such that the region $B_{\Ho}(r,x_0) \times \R^{2}$ is geodesically compact.

\item \label{cond-3}  The sequence of minimizing geodesics \(\{c_n\}_{n\in\mathbb N}\), defined by Eq. \eqref{eq:shot-con}, is precisely normal in the sense of Definition \ref{def:str-nor-sec}. Moreover, no subsequence of $\{c_n\}_{n\in\mathbb N}$ converges to a line-geodesic.

\item \label{cond-4} The map $\Theta_2: Homc^+(x_0) \to \R$ is one-to-one.

\end{enumerate}
Then $c_h$ is a metric line in $\MS$. 
\end{Theorem}

\begin{Remark}
    Given a magnetic space $\MS$ and a geodesic $c_h$, one must verify conditions  (1)--(4) to apply Theorem \ref{the:metrtic-lines-method}. Condition (3) is the most delicate hypothesis of Theorem \ref{the:metrtic-lines-method}, since a sequence of normal geodesics does not, in general, converge to a normal geodesic. For instance, in the case of the Engel-type group, Lemmas \ref{lem:cond-1}, \ref{lem:geod-com-Ent}, \ref{c},  and \ref{lem:theta-r-bounded} verify these conditions.

     Lemma \ref{lemma:abn-geo-mag} implies that condition (3) is automatically satisfied when $F$ is a polynomial of degree $2$ with a unique critical point. Consequently, Theorem \ref{the:metrtic-lines-method} can potentially be applied to establish Conjecture \ref{conj} for homoclinic geodesics in Carnot groups satisfying the assumptions of Section \ref{sub-sub-sec:exp-coor} for which the associated polynomial $F$ has a unique critical point.
    \end{Remark}

The following section proves Theorem \ref{the:metrtic-lines-method}. 

\subsubsection{Proof of Theorem \ref{the:metrtic-lines-method}}
Throughout this section, we assume the hypotheses of Theorem \ref{the:metrtic-lines-method}.


\begin{lemma}
\label{cor:T_h}
There exists \(T_h^*>0\) such that, for every \(t>T_h^*\), the following properties hold:
\begin{itemize}
    \item $y_h(-t) <0 < y_h(t)$.

    \item $x_h(-t),x_h(t)\in B_\Ho(r,x_0)$, where $B_\Ho(r,x_0)$ denotes the ball defined in  condition (\ref{cond-2}).
\end{itemize}
\end{lemma}

\begin{proof}
Such a \(T_h^*\) exists because the homoclinic geodesic \(c_h\) satisfies 
$$\lim_{t \to \pm \infty}y_h(t) = \pm\infty\;\; \text{and} \;\; \lim_{t \to \pm\infty}x_h(t) = x_0.$$
\end{proof}

The first step in constructing the set \(\operatorname{Min}(K,T)\) is to show that each \(c_n\) intersects a fixed compact subset of $\MS$.

\begin{Prop}\label{prp:bounded-ini-con}
Let $T^*_h$ be as in Lemma \ref{cor:T_h}. There exists a compact subset $K_0 \subset \MS$ such that, for every \(n>T_h^*\),  there exists $t^*_n \in (0,T_n)$ satisfying $c_n(t^*_n)\in K_0$ .
\end{Prop}

\begin{proof}
Let $n \in \mathbb{N}$ satisfy  $T^*_h < n$. By Lemma 3.7 and the intermediate value theorem, there exists \(t_n^*\in(0,T_n)\) such that \(y_n(t_n^*)=0\). Lemma \ref{lem:uniform-bound-c-s} and Eq. \eqref{eq:cots-fuc-cond} imply that $\|\operatorname{Cost}(c_{n},[0,T_n])\|_{\sup} < \Theta_1([c_h])$, where condition (\ref{cond-1}) guarantees that $\Theta_1([c_h])$ is finite. For simplicity, let $C:=\Theta_1([c_h])$. It follows that $(c_n,\T_n) \in \operatorname{Com}(r,x_0,C)$. Therefore, Condition (\ref{cond-2}) guarantees the existence of a compact set $K_{C}$ such that each $x_n([0,T_n])\subset K_{C}$ for all $n>T^*_h$; in particular, $x_n(t_n^*) \in K_{C}$.

Claim. The sequence \(\{z_n(t_n^*)\}_{n\in\mathbb N}\) is uniformly bounded. By the definition of the cost map,
\begin{equation}\label{eq:prp-ini-time}
\begin{split}
\Delta z(c_n,[0,t^*_n]) & =  \Delta y(c_n,[0,t^*_n]) - \operatorname{Cost}_y(c_n,[0,t^*_n]),  \\
\Delta z(c_{h},[-n,0]) & = \Delta y(c_{h},[-n,0]) - \operatorname{Cost}_y(c_{h},[-n,0]). \\
\end{split}
\end{equation}
By Eq. \eqref{eq:ini-cond-can}, $0 = y_h(0) = z_h(0)$, while the choice of \(t_n^*\) yields \(y_n(t_n^*)=0\). Moreover, the endpoint conditions give
$$z_n(0) = z_{h}(-n),\quad \Delta y(c_n,[0,t^*_n]) =\Delta y(c_{h},[-n,0]) .$$ 
Using these identities and subtracting the two equalities in Eq. \eqref{eq:prp-ini-time}, we obtain
\begin{equation}
\begin{split}
| z_n(t^*_n) |  & = | \Delta z(c_n,[0,t^*_n]) -  \Delta z(c_{h},[-n,0])| \\
                & \leq |\operatorname{Cost}_y(c_n,[0,t^*_n])| + |\operatorname{Cost}_y(c_{h},[-n,0])| . \\
\end{split}
\end{equation}
Lemma \ref{lem:uniform-bound-c-s} implies $|\operatorname{Cost}_y(c_h,[-n,0])|\leq C$. Furthermore, Lemma \ref{lem:uniform-bound-c-n} below shows that $
|\operatorname{Cost}_y(c_n,[0,t^*_n])|<C_1$ for all $n>T^*_h$. Define the compact set 
$$ K_0 := K_{C} \times [-1,1] \times [-C_2,C_2], \;\text{where}\;C_2 = C+C_1.$$
Therefore, $c_n(t_n^*)\in K_0$ for all $n>T^*_h$.
\end{proof}

The following result establishes the existence of the constant $C_1$.

\begin{lemma}\label{lem:uniform-bound-c-n}
There exists a constant $C_1$ such that $|\operatorname{Cost}_y(c_n,[0,t^*_n])| < C_1$ for all $n>T^*_h$.
\end{lemma}

\begin{proof}
    We use the notation introduced in the proof of Lemma \ref{lem:com-converge}. Write
\[
G_{\nu_n}^F=a_n+b_nF
\]
and set
\[
u_n(t):=1-G_{\nu_n}^F(x_n(t)),
\qquad
d_n:=\frac{a_n+b_n-1}{b_n}.
\]
Therefore Eq. \eqref{eq:cost-y-double-lemma} holds. 

By condition (\ref{cond-3}) of Theorem \ref{the:metrtic-lines-method}, no subsequence of $\{c_n\}$ converges to a line-geodesic. Thus, no subsequence of $\{b_n\}$ can converge to $0$.  Hence,  \(b_n\) is bounded away from zero. Moreover, as in the proof of Lemma \ref{lem:com-converge},
\[
0\le u_n(t)\le 2,
\qquad
\int_{\T_n}u_n(t)dt
=\operatorname{Cost}_t(c_n,\T_n)\le C.
\]
In particular, for every interval \(I\subset \T_n\),
\[
\int_Iu_n(t)dt\le C,
\qquad
\int_Iu_n(t)^2dt\le 2C.
\]
Using Eq. \eqref{eq:cost-y-double-lemma}, we obtain
\begin{equation*}
    \begin{split}
        |\operatorname{Cost}_y(c_n,I)
-d_n|I|
|&\leq \left|\frac1{b_n}-d_n\right|\int_{I}u_n(t)dt
+\frac1{|b_n|}\int_{I}u_n(t)^2dt\\
&\leq \left|\frac1{b_n}-d_n\right|C
+\frac{C}{|b_n|}.
    \end{split}
\end{equation*}

Therefore, since \(b_n\) is bounded away from zero and \({d_n}\) is bounded, there exists \(C'>0\), independent of \(n\) and of the interval \(I\subset T_n\), such that
\begin{equation}\label{eq:proof-lema-cot-y-bound}
    \left|
\operatorname{Cost}_y(c_n,I)-d_n|I|
\right|\le C'.
\end{equation}
Applying this estimate to \(I=[0,T_n]\) and using the uniform boundedness of \(|\operatorname{Cost}_y(c_n,[0,T_n])|\), we obtain
\[
|d_n|T_n\le C''
\]
for some \(C''>0\). Since \(0< t_n^*< T_n\), it follows that
\[
|d_n|t_n^*\le C''.
\]
Applying Eq. \eqref{eq:proof-lema-cot-y-bound} once more, now with \(I=[0,t_n^*]\), yields
\[
\left|\operatorname{Cost}_y(c_n,[0,t_n^*])\right|
\le |d_n|t_n^*+C'
\le C'+C''=:C_1.
\]
Thus, \(\{\operatorname{Cost}_y(c_n,[0,t_n^*])\}_{n\in\mathbb N}\) is uniformly bounded by $C_1$.
\end{proof}


 Define the shifted interval $\mathcal{T}_n := [-t^*_n,T_n-t^*_n]$ and reparametrize $c_n$ by $\tilde{c}_n(t):= c_n(t+t^*_n)$. Then $\tilde{c}_n|_{\mathcal{T}_n}$ is minimizing and $\tilde{c}_n(0) \in K_0$ for all $n>T^*_h$.  
 
\begin{lemma}\label{cor:sub-seq-time-int}
There exists a subsequence $\{\mathcal{T}_{n_j}\}$ such that $\mathcal{T}_{n_j} \subset \mathcal{T}_{n_{j+1}}$.
\end{lemma}

\begin{proof}
If $\tilde{c}_{n}$ is a geodesic minimizing on $\mathcal{T}_n $, then 
$$ dist_{\MS}(\tilde{c}_{n}(0),\tilde{c}_{n}(-t^*_n)) = t^*_n, \;\text{and}\; dist_{\MS}(\tilde{c}_{n}(0),\tilde{c}_{n}(T_n-t^*_n)) = T_n-t^*_n.$$ 
The sequence \(\{\widetilde c_n(0)\}\) is bounded, whereas \(\{\widetilde c_n(-t_n^*)\}\) and \(\{\widetilde c_n(T_n-t_n^*)\}\) are unbounded. Therefore,
$$ \lim_{n \to \infty} t^*_n  = \lim_{n \to \infty} T_n-t^*_n = \infty.$$
Therefore, we may select a subsequence $\{\mathcal{T}_{n_j}\}$ such that $\mathcal{T}_{n_j} \subset \mathcal{T}_{n_{j+1}}$. 
\end{proof}

For simplicity, relabel the subsequence $\{\mathcal{T}_{n_j}\}$ by $\{\T_n\}$. 

\begin{lemma}\label{lem:Min-seq}
For every \(N\in\mathbb N\) such that \(N>T_h^*\), there exists a compact set $K_N \subset \MS$ such that $\tilde{c}_n\in Min(K_N,\mathcal{T}_N)$ for all $n > N$.
\end{lemma}

\begin{proof}
Fix \(N\in\mathbb N\) such that $N>T^*_h$. If $n>N$, then $\mathcal{T}_N \subset \mathcal{T}_{n}$ and the geodesic $\tilde{c}_{n}$  minimizes arc length on the interval $\mathcal{T}_N$. Moreover, $\|\Delta x(\tilde{c}_{n},\mathcal{T}_N)\|,|\Delta y(\tilde{c}_{n},\mathcal{T}_N)|\leq T_N$, since $T_N$ is the length of the interval $\mathcal{T}_{N}$, and by construction $\|\dot{\tilde{x}}_n\|_{\Ho} \leq 1$ and $|\dot{\tilde{y}}_n| \leq 1$. Applying Eq. \eqref{eq:ode-y-z}, it follows that
$$ | \Delta z(\tilde{c}_{n},\mathcal{T}_N)|  \leq \int_{\mathcal{T}_N} |F(\tilde{x}_n(t))| dt \leq T_N \max_{x \in K_{C} + B_{\Ho}(0,T_N)} |F(x)| =: C_z.  $$
If $K_N := K_0 + [-T_N,T_N]^n \times [-T_N,T_N] \times [-C_z,C_z]$, then $\tilde{c}_{n}(\mathcal{T}_N) \subset K_N$.
\end{proof}

We have shown that $\tilde{c}_n\in Min(K_N,\mathcal{T}_N)$. By Proposition \ref{prp:seque-comp} and a diagonal argument, there exists a convergent subsequence $\{\tilde{c}_{n_j}\}$ converging to  $c_{\infty}$. Condition (\ref{cond-3}) implies that $c_{\infty}$ is a normal geodesic. The following lemma shows that $[c_h] = [c_{\infty}]$. 
 
\begin{lemma}\label{lem:uniq-G}
The limit geodesic $c_\infty$ satisfies $[c_\infty] = [c_h]$.
\end{lemma} 

\begin{proof}
Consider the geodesic $c_{\infty}$ and its reduced trajectory $x_\infty$. The asymptotic conditions in Eq. \eqref{eq:asymp-cond-h} imply that $x_\infty(t) \to x_0$ as $t \to \pm \infty$. Therefore, $c_\infty$ is homoclinic. According to condition (\ref{cond-3}), $c_{\infty}$ is a normal geodesic and therefore has an associated polynomial $G_{\nu}^F\in Pen_{F}$. By construction, $\{(\tilde{c}_{n_j},\T_{n_j})\}_{j\in\mathbb{N}} \subset \operatorname{Com}(r,x_0,C)$, and Lemma \ref{lem:com-converge} implies that $G_{\nu}^F(x_0) = 1$. Thus, $[c_{\infty}]\in Homc^+(x_0)$. By Eq. \eqref{eq:asymp-cond-theta}, \(\Theta_2([c_\infty])=\Theta_2([c_h])\). Since \([c_\infty],[c_h]\in Homc^+(x_0)\), Condition (\ref{cond-4}) implies that \([c_\infty]=[c_h]\).
\end{proof}

\begin{proof}[Proof of Theorem \ref{the:metrtic-lines-method}]
Since $c_{\infty}$ and $c_{h}$ belong to the same equivalence class, there exists an isometry $\varphi\in Iso^*(\MS)$ such that $c_{h} = \varphi\circ c_{\infty}$. Since \(c_h([-T,T])\) is bounded, we can choose $N$ sufficiently large so that, setting $K := K_N$ and  $\mathcal{T} := \mathcal{T}_N$, we have
$$c_{h}([-T,T]) \subset \varphi (c_{\infty}(\mathcal{T}))\;\text{and}\; c_{\infty}\in Min(K,\mathcal{T}).$$ Lemma \ref{lem:sub-submersion} implies that $c_{h}$ is minimizing on $[-T,T]$. Since $T>0$ is arbitrary,  $c_{h}$ is a metric line in $\MS$.
\end{proof}

\section{The Engel-type Group}\label{sec:Eng-type}

Let $\Eng(n)$ denote the Engel-type group, and let $\mathfrak{eng}(n)$ be its Lie algebra. The Lie algebra $\mathfrak{eng}(n)$ has basis  $\{E_1,\dots, E_n, E_0^{\Laa},\dots,E_{n+1}^{\Laa}\}$, with the following nontrivial Lie bracket relations:
\begin{equation}\label{eq:Eng-alg}
E_{i}^{\Laa} = [E_i,E_{0}^{\Laa}] \;\;\text{and} \;\; E_{n+1}^{\Laa} = [E_i,E_{i}^{\Laa}] .
\end{equation} 
Eq. \eqref{eq:Eng-alg} implies that $\{E_1,\dots, E_n,E_0^{\Laa}\}$ and $\{E_0^{\Laa},\dots,E_{n+1}^{\Laa}\}$ are bases for $\V_1$ and $\Laa$, respectively. By definition of the subspace $\Lav:= \Lah\cap\Laa$, it follows that $\dim\Lav = 1$ and $\{E_1,\dots, E_n\}$ is a basis for $\Lah$. Thus, Eq. \eqref{eq:Eng-alg} implies that $\Eng(n)$ is a \ma group and satisfies the assumptions of Section \ref{sub-sub-sec:exp-coor}. In \cite{donne2020semigenerated}, E. Le Donne and T. Moisala studied the \sR structure of $\Eng(n)$. Moreover, in \cite{le2008cornucopia}, E. Le Donne and F. Tripaldi classified the Carnot groups in low dimension; in this context, they denoted the group $\Eng(2)$ by $\Na6$.

Let $g = \Phi(x,\theta)$ denote the coordinates defined in Section \ref{sub-sub-sec:exp-coor}, where $x = (x_1,\dots,x_n)$ and $\theta = (\theta_0,\theta_1,\dots,\theta_{n+1})$. In these coordinates, the following left-invariant vector fields define a frame for the distribution $\Di$:
\begin{equation*}
\begin{split}
X_i & = \frac{\partial}{\partial x_i} \;\;\text{for}\;i = 1,\dots,n, \\
\; Y & = \frac{\partial}{\partial \theta_0} + x_1 \frac{\partial}{\partial \theta_1} + \dots + x_n \frac{\partial}{\partial \theta_n} + \frac{1}{2} ( x_1^2+\dots+x_n^2) \frac{\partial}{\partial \theta_{n+1}}.
\end{split}
\end{equation*}

Therefore, the projection is given by $\pi(x,\theta) = (x,\theta_0)$, and the \sR metric in the coordinates $(x,\theta)$ is given by 
$$ds^2_{\Eng(n)} = \Big(\sum_{i=1}^n dx_i^2 + d\theta_0^2\Big)_\D.$$

The following lemma characterizes the abnormal extremals in the Engel-type group.
\begin{lemma}\label{lemma:abn-geo-Eng}
The Engel-type group $\Eng(n)$ has two types of abnormal extremals:

\textbf{(Vertical)} An abnormal extremal $\gamma$ is called  vertical if $\dot{\gamma}(t)\in \mathrm{span}\{ Y\}_{\gamma(t)}$. A vertical abnormal extremal is a geodesic if and only if it is a vertical line.

\textbf{(Horizontal)} An abnormal extremal $\gamma$ is called  horizontal if $\dot{\gamma}(t)\in \mathrm{span}\{ X_1,\ldots,X_n\}_{\gamma(t)}$.  A horizontal abnormal extremal is a geodesic if and only if it is a horizontal line-geodesic.

Therefore, \(\operatorname{Eng}(n)\) has no strictly abnormal geodesics.
\end{lemma}

The terms horizontal and vertical have the same meaning as in Lemma \ref{lemma:abn-geo-mag}. The proof of Lemma \ref{lemma:abn-geo-Eng} is provided in Appendix \ref{sub-AP:ab-geo}.

\subsection{Reduced Hamiltonian}\label{sub-sec:red-ham-Eng}

If $\mu = (a_0,\dots,a_{n+1})\in \Laa^*$, then Eq. \eqref{eq:red-Ham} implies that the reduced Hamiltonian $H_{\mu}$ has the form
\begin{equation}\label{eq:fund-eq-F-Eng}
H_{\mu}(p_x,x) = \frac{1}{2} \sum_{i=1}^n p_{x_i}^2 + \frac{1}{2} ( a_0 + \sum_{i=1}^n a_i x_i  +\frac{a_{n+1}}{2}  \sum_{i=1}^n x_i^2 )^2.
\end{equation}

When $a_{n+1} = 0$, the reduced mechanical system has a quadratic potential. The study of quadratic potentials is well established in small-oscillation theory \cite{arnol2013mathematical,Landau}. When $a_{n+1} \neq 0$, we introduce the following change of coordinates and constants to obtain a more convenient expression for the reduced Hamiltonian:
\begin{equation*}
    \begin{split}
        (\tilde{x}_1,\dots,\tilde{x}_n) & = (\frac{a_1}{a_{n+1}} +x_1,\dots, \frac{a_n}{a_{n+1}} +x_n),\\
        (\alpha,\beta) & := (a_0 - \frac{1}{2a_{n+1}}\sum_{i=1}^n a_i^2, \frac{a_{n+1}}{2}).
    \end{split}
\end{equation*}
If  $r := \|\tilde{x}\|_{\Ho}$, where $\|\cdot\|_{\Ho}$ denotes the Euclidean norm in $\Ho$, then the reduced Hamiltonian takes the form 
\begin{equation}\label{eq:an-harm-osci}
H_{\mu}(p_{\tilde{x}},\tilde{x}) = \frac{1}{2} \sum_{i=1}^n p_{\tilde{x}_i}^2 + \frac{1}{2} F_{\mu}^2(r), \;\;\text{where} \;\; F_{\mu}(r) = \alpha + \beta r^2.
\end{equation}
For the remainder of the paper, we denote $(p_{\tilde{x}},\tilde{x})$ by $(p_x,x)$.

If $(x(t),p_x(t))$ is a solution of an $n$-degree-of-freedom Hamiltonian system with radial potential, then the curve $x(t)$ lies in a subspace of dimension at most two. The following lemma formalizes this idea.

\begin{lemma}\label{lem:planar-an-harmonic}
Let $\iota: T^*\R^2 \to T^*\Ho$ be the inclusion given by
$$ \iota(p_{x_1},p_{x_2},x_1,x_2) :=  (p_{x_1},p_{x_2},0,\dots,0,x_1,x_2,0,\dots,0).$$
 \begin{itemize}
     \item[(a)]  If $(p_x(t),x(t))$ is a solution to the Hamiltonian system defined by \eqref{eq:an-harm-osci} with the initial condition $(p_x(0),x(0))\in \iota (T^*\R^2)$, then $(p_x(t),x(t))\in \iota (T^*\R^2)$ for all $t$.

    \item[(b)] For \(Q\in\mathrm{SO}(n)\), let \(\Psi_Q\) denote the map induced by the right action of \(Q\) on \(H\), and let \(T^*\Psi_Q\) denote its cotangent lift to \(T^*H\). Then every solution $(p_x(t),x(t))$ to the Hamiltonian system defined in Eq. \eqref{eq:an-harm-osci} can be expressed as $(p_x(t),x(t)) = T^*\Psi_Q(\tilde{p}_x(t),\tilde{x}(t))$, where $(\tilde{p}_x(t),\tilde{x}(t))\in \iota (T^*\R^2)$ is a solution as described in part (a).
 \end{itemize}

\end{lemma}

\begin{proof}
First, we prove that if $(p_x(0),x(0))\in \iota (T^*\R^2)$, then $(p_x(t),x(t))\in \iota (T^*\R^2)$ for all $t$. By construction, the Hamiltonian function is invariant under the action of $SO(n)$ on $T^*\Ho$. Thus, the $SO(n)$-action admits a momentum map $J:T^*\Ho \to \mathfrak{so}^*(n)$. Using the Killing form, $\mathfrak{so}^*(n)$ can be identified with $\mathfrak{so}(n)$. Under this identification, $J:T^*\Ho \to \mathfrak{so}(n)$ and a standard computation gives $J(p_x,x) = p_x \wedge x$. Thus, the skew-symmetric matrix associated with the two-form $p_x \wedge x$ has at most rank 2. Moreover, conservation of the momentum map gives $0 \neq p_x(0) \wedge x(0) = p_x(t) \wedge x(t)$. Therefore, the vectors $p_x(t)$ and $x(t)$ lie in a fixed two-dimensional plane for all $t\in \R$.

Consider an arbitrary initial condition $(p_x(0),x(0))$. There exists a matrix $Q^{-1}\in SO(n)$ such that $T^*\Psi_{Q^{-1}}(p_x(0),x(0))\in \iota (T^*\R^2)$. Let $(\tilde{p}_x(t),\tilde{x}(t))$ be the solution to the reduced system with initial condition $\Psi^*_{Q^{-1}}(p_x(0),x(0))$. Since the Hamiltonian is \(SO(n)\)-invariant, $T^*\Psi_Q$ maps solutions of the reduced system to solutions  \cite[Section 6]{arnold1992ordinary}. Thus, $(p_x(t),x(t)) = T^*\Psi_Q(\tilde{p}_x(t),\tilde{x}(t))$ is a solution to the reduced system. By construction, $(p_x(t),x(t))$ has initial condition $(p_x(0),x(0))$, so $(p_x(t),x(t))$ is the desired solution. 
\end{proof}

Lemma \ref{lem:planar-an-harmonic} reduces the study and classification of the dynamics of $H_{\mu}$ to the planar anharmonic oscillator.

\subsubsection{The Planar Radial Anharmonic Oscillator}

This section studies the planar anharmonic oscillator. More specifically, we restrict the reduced Hamiltonian defined in Eq. \eqref{eq:an-harm-osci} to the submanifold $\iota(T^*\R^2)$ and express it in polar coordinates:
\begin{equation}
H_{\mu}(p_x,x)|_{\iota(T^*\R^2)} = H_{\mu}(p_r,p_{\theta},r,\theta) = \frac{1}{2}(p_r^2 + \frac{p_{\theta}^2}{r^2}) + \frac{1}{2}F^2_{\mu}(r).
\end{equation}

Because the potential is radial, the coordinate $\theta$ is cyclic; thus, $p_{\theta}$ is a constant of motion. Fixing $p_{\theta} = \ell$ reduces the Hamiltonian system to a one-degree-of-freedom Hamiltonian system of the form
\begin{equation}\label{eq:planar-an-harm-osc}
H_{(\mu,\ell)}(p_r,r) = \frac{1}{2}p^2_r + \frac{1}{2}V_{(\mu,\ell)}(r),\;\;\text{where}\;\; V_{(\mu,\ell)}(r) := \frac{\ell^2}{r^2} + F^2_{\mu}(r).
\end{equation}
By fixing the energy level $H_{(\mu,\ell)} = \frac{1}{2}$ and using Hamilton's equation $p_r = \dot{r}$, the radial dynamics reduce to quadrature:
\begin{equation}
t-t_0 = \int_{r_0}^r \frac{dr}{\sqrt{1-V_{(\mu,\ell)}(r)}}, \;\; \text{where}\;\; r_0 = r(t_0).
\end{equation}

The effective potential is non-negative and unbounded. The conservation of energy $H_{(\mu,\ell)}(p_x(t),r(t)) = \frac{1}{2}$ implies that the Hill region is compact. The radial coordinate \(r(t)\) lies in a connected component of this region, which is a closed interval called a radial Hill interval.  More precisely, an interval $\mathcal{R} = [r_{min},r_{\max}]$ is called a radial Hill interval of the effective potential $V_{(\mu,\ell)}(r)$ if $V_{(\mu,\ell)}(r_{min}) = V_{(\mu,\ell)}(r_{max}) = 1$ and $V_{(\mu,\ell)}(r) < 1$ for all $r\in (r_{min},r_{\max})$.

The following lemma characterizes the relative equilibria of the planar anharmonic oscillator.


\begin{lemma}\label{lemma:equi-point-an-har}
The relative equilibria of the reduced system defined by the Hamiltonian in Eq. \eqref{eq:planar-an-harm-osc} on the energy level $H_{(\mu,\ell)} = \frac{1}{2}$  are as follows:
\begin{itemize}
    \item If $\ell = 0$, there exists a unique relative equilibrium point $(p_r,r) = (0,0)$ if and only if $\alpha = \pm 1$. If, in addition, \(\alpha\beta<0\), this point has a homoclinic orbit with radial Hill interval $ \mathcal{R} =[0,\sqrt{-\frac{2\alpha}{\beta}}]$.

    \item If $\ell \neq  0$,  there exists a unique relative equilibrium $(p_r,r) = (0,r^*)$ for suitable values of the parameters $(\mu,\ell)$. In this case, the radial Hill interval is a singleton $\{r^*\}$.
\end{itemize}
\end{lemma}

\begin{proof}
    The reduced equations are
    $$\dot{r} = p_r, \;\;\text{and}\;\; \dot{p}_r = -\frac{2\ell^2}{r^3} + 2\beta r(\alpha+\beta r^2).$$
On the energy level \(H_{(\mu,\ell)}=\frac12\), an equilibrium satisfies $p_r = 0$ and  $V_{(\mu,\ell)}(r) = 1$. If $\ell = 0$, then $V_{(\nu,0)}(r) = 1$ implies  $(\alpha+\beta r^2)^2=1$, indicating that $\dot{p}_r = 0$ if and only if $r=0$. Therefore, the point $(p_r,r)$ is a relative equilibrium if and only if $(p_r,r) = (0,0)$, where $F_{\mu}(r) = \pm (1-|\beta| r^2)$.

If $\ell \neq 0$, energy conservation together with the condition $\dot{p}_r = 0$ yields
$$1- \frac{2\ell^2}{r^2} = (\alpha+\beta r^2)^2 ,\;\text{and}\;\; \frac{\ell^4}{\beta ^2 r^8} = 1- \frac{2\ell^2}{r^2}.$$
The left-hand side of the preceding equation is monotonically decreasing for $r>0$, and the right-hand side is monotonically increasing. Therefore, there is a unique $r^*$ such that $(p_r,r) = (0,r^*)$ is a relative equilibrium. At the relative equilibrium, \(r(t)\equiv r^*\); hence the radial Hill interval is the singleton \(R=\{r^*\}\). 
\end{proof}
 
\begin{Remark}
The following observations concern Lemma \ref{lemma:equi-point-an-har}:
\begin{itemize}
    \item The case $\ell = 0$, with \(\alpha=\pm1\) and $\alpha\beta<0$,  gives rise to a family of homoclinic geodesics. This family generalizes the family of metric lines in the Engel group, where the momentum $\mu$ determines the polynomial $F_{\mu}(x) = \pm (1-\beta x^2)$, as detailed in \cite[Theorem B]{bravo2022metric}.
    
    \item  The geodesic corresponding to the relative equilibrium with $\ell \neq 0$ has constant radial coordinate $r(t) = r^*$. Thus, this geodesic is classified as $r$-periodic. 
\end{itemize}    
\end{Remark}

\subsubsection{Classification of Normal \SR Geodesics in the Engel-type Group}\label{subsub:clas-sr-geo}


Consider a normal geodesic $\gamma$ that is not of line type and has momentum $\mu \neq 0$. Then $\gamma$ belongs to one of the following classes:

\textbf{Small-oscillation.} A geodesic $\gamma$ is classified as a small oscillation if $a_{n+1} = 0$.

\textbf{$r$-periodic.} A geodesic $\gamma$ is classified as $r$-periodic if $a_{n+1} \neq 0$ and the radial dynamics are periodic.

\textbf{$r$-homoclinic.} A geodesic $\gamma$ is classified as $r$-homoclinic if $a_{n+1} \neq 0$ and the radial dynamics exhibit a homoclinic orbit.


\begin{Prop}\label{prp:no-met-lin-Eng}
Geodesics in $\Eng(n)$ of the small-oscillation and $r$-periodic types are not metric lines.
\end{Prop}

\begin{proof}
By construction, the $r$-periodic geodesics are confined to a regular region of the potential $F_{\mu}(r)$. Hence, by Corollary \ref{Cor:turn-back}, they are not metric lines. 

Now consider a small-oscillation geodesic $\gamma$. We show that $\gamma$ does not satisfy the condition in Proposition \ref{prop:met-char}. The reduced Hamiltonian has a quadratic potential, so we can apply the theory of small oscillations. Let $B$ denote an $n\times n$  symmetric matrix associated with this quadratic potential. By construction, $B$ has rank one. Consequently, the eigenspace corresponding to the eigenvalue \(0\) has dimension \(n-1\), while the unique nonzero eigenvalue is $\omega^2$. After a change of variables, the reduced Hamiltonian $H_{\mu}$ decomposes into the sum of two Hamiltonian functions:  an $n-1$-degree-of-freedom system for a free particle $H_{free}(p_{x_1}, \dots, p_{x_{n-1}})$ and a one-degree-of-freedom system for a harmonic oscillator $H_{osc}(p_{x_n},x_n)$ defined by
$$ H_{free} := \frac{1}{2}(p_{x_1}^2 + \cdots + p^2_{x_{n-1}}) \;\;\text{and}\;\; H_{osc} := \frac{1}{2}(p_{x_n}^2 + \omega^2 x_n^2). $$ 

If $H_{osc} = 0$, which is equivalent to $H_{free} = \frac{1}{2}$, then  $\pi\circ\gamma$ is a straight line in $\V$ and the geodesic $\gamma$ is a line-geodesic. Consequently, it suffices to consider the case  $H_{osc} \neq 0$, which implies that $H_{free} < \frac{1}{2}$. By left invariance, we may assume $\gamma(0) = 0$. In this case, 
$$ \dot{\gamma}(t) = (p_{x_1} X_1+ \cdots + p_{x_{n-1}} X_{n-1} + \sqrt{2H_{osc}}\cos(\omega t) X_{n} + \frac{\sqrt{2H_{osc}}}{\omega}\sin(\omega t)Y)|_{\gamma(t)}. $$
Note that  $1 = \|\dot{\gamma}\|^2 = 2 H_{free} + 2H_{osc}$. Let $v = (v_1,\cdots,v_n,v_{n+1})$ be an arbitrary unit vector in $\V_1$. Then
\begin{equation*}
\begin{split}
 |\lim_{t \to \infty} \frac{1}{2t} \int_{-t}^t  \langle v,(L_{\gamma^{-1}(s)})_*\dot{\gamma}(s)\rangle ds | & = |p_{x_1} v_1+ \cdots + p_{x_{n-1}} v_{n-1}| \\
  & \leq \sqrt{2H_{free}}   < 1.
 \end{split}
\end{equation*}
Therefore, $\gamma$ does not satisfy the condition in Proposition \ref{prop:met-char}.
\end{proof}

\subsubsection{Radial Geodesics}\label{sub-sec:radial-geo}

\begin{Prop}\label{prp:act-eng-alg}
The following formulas define an \(\mathrm{SO}(n)\)-action on $\mathfrak{eng}(n)$:
$$ \tilde{E}_j = \sum_{i=1}^n Q_{ji} E_i,\;\; \tilde{E}_0^{\Laa} = E^{\Laa}_0, \;\; \tilde{E}_j^{\Laa} = \sum_{i=1}^n Q_{ji} E_i^{\Laa},\;\; \tilde{E}_{n+1}^{\Laa} = E^{\Laa}_{n+1}, $$
where $Q := (Q_{ij})\in SO(n)$. Moreover, this action on $\mathfrak{eng}(n)$ induces an isometric $SO(n)$-action on $\Eng(n)$, denoted by $\psi_Q$, for \(Q\in SO(n)\). If $(x,\theta)$ denotes the exponential coordinates of the second type, then the action  $\psi_Q (x,\theta) = (\tilde{x},\tilde{\theta})$ is given by
\begin{equation*}
\tilde{\theta_0} = \theta_0, \;\; \tilde{\theta}_j = \sum_{i=1}^n Q_{ji}\theta_i, \;\; \tilde{\theta}_{n+1} = \theta_{n+1}, \;\; \tilde{x}_j = \sum_{i=1}^n Q_{ji}x_i.
\end{equation*}
\end{Prop}

\begin{proof}
    A direct computation shows that the $SO(n)$-action preserves the Lie bracket relations of $\mathfrak{eng}(n)$. Since $\Eng(n)$ is connected and simply connected, the \(\mathrm{SO}(n)\)-action induces an action on $\Eng(n)$ \cite[Theorem 5.6]{hall2003lie}. Moreover, the action preserves the inner product on \(\V_1\); hence the induced action on \(\Eng(n)\) is isometric.
\end{proof}

\begin{lemma}\label{lem:Eng(2)-geo}
Let $\iota_c:\Eng(2) \to \Eng(n)$ be the inclusion given by
$$ \iota_c(x_1,x_2,\theta_0,\theta_1,\theta_2,\theta_{n+1}) = (x_1,x_2,0,\dots,0,\theta_0,\theta_1,\theta_2,c_3,\dots,c_n,\theta_{n+1}), $$
where $c = (c_3,\dots,c_n)$.
\begin{itemize}
    \item[(a)] If $\gamma$ is a normal \sR  geodesic in $\Eng(n)$ with momentum $\mu$ satisfying $a_{n+1} \neq 0$ such that $\gamma(0)\in \iota_c(\Eng(2))$ and $\dot{\gamma}(0) \in T_{\gamma(0)} \iota_c(\Eng(2))$, then $\gamma(t)$ lies in $\iota_c(\Eng(2))$ for all $t$.

    \item[(b)] Moreover, every normal geodesic in $\Eng(n)$ with momentum $\mu$ satisfying $a_{n+1} \neq 0$ has the form $\gamma = \psi_Q\circ \gamma_0$, where $\psi_Q$ is  defined in Proposition \ref{prp:act-eng-alg} and $\gamma_0$ is a normal geodesic in $\iota_c(\Eng(2))$.
\end{itemize}
\end{lemma}

The proof of Lemma \ref{lem:Eng(2)-geo} is identical to that of Lemma \ref{lem:planar-an-harmonic}. Consequently, it suffices to prove Theorem \ref{the:main-1} for $\Eng(2)$.

\begin{Prop}\label{prp:Euler-Elastica-Eng}
Let $\gamma$ denote a normal geodesic in $\Eng(n)$ with momentum $\mu$ such that $a_{n+1} \neq 0$ and whose reduced dynamics exhibit zero angular momentum. If $\eta := \pi\circ\gamma$, where $\pi:\Eng(n) \to \R^{n+1}$ denotes the canonical projection, then $\eta$ lies in a two-dimensional plane and solves the Euler-Elastica problem.  
\end{Prop}

The Euler-Elastica problem can be characterized as follows. A plane curve solves the Euler-Elastica problem if, at each point, its curvature is proportional to the oriented distance from that point to a specified line called the directrix.

\begin{proof}
According to Lemma \ref{lem:Eng(2)-geo},  it suffices to consider a normal geodesic $\gamma$  in $\Eng(2)$. By applying an isometry, we may assume without loss of generality that $\gamma(0) = (x_1(0),0,0,0,0,0)$ for some $x_1(0)\in \R$. Let $(p_x(t),x(t))$ denote the solution to the reduced system corresponding to $\gamma$. Zero angular momentum implies that the reduced velocity is radial. Therefore, the corresponding reduced solution takes the form $(p_x(t),x(t)) = (p_{x_1}(t),0,x_1(t),0)$. If  $F_{\mu}(x)$ is the polynomial from Eq. \eqref{eq:fund-eq-F-Eng}, then   $F_{\mu}(x(t)) = a_0+a_1x_1(t)+\frac{a_3}{2} x_1^2(t)$. 


By construction, the projected curve $\eta := \pi\circ\gamma$ lies in the $\{x_2=0\}\subset \R^3$. In addition, Theorem \ref{the:background} implies that $\dot{\eta}(t) = (p_{x_1}(t),0,F_{\mu}(x(t)))$ and $\|\dot{\eta}(t)\|  = 1$. The reduced Hamilton equations yield
$$ \ddot{\eta}(t) = \frac{\partial F_{\mu}(x(t))}{\partial x_1} N(t) , \;\; \text{where}\;\; N(t) = (-F_{\mu}(x(t)),0,p_{x_1}(t)). $$ 
It follows that $\{\dot{\eta}(t),N(t)\}$ is a positively oriented basis for the plane $\{x_2=0\})$. Thus, the curvature of $\eta$ is  $ a_3 x_1(t)+a_1$, and hence is proportional to the oriented distance from $\eta(t)$ to the directrix defined by  $x_1 = - \frac{a_1}{a_3}$.
\end{proof}


\subsection{The Magnetic Space for the Engel-Type Group}\label{sub-sec:mag-spa-Eng}
Consider a candidate metric line $\gamma_h$ whose momentum determines the polynomial $F(x)= 1 - \beta r^2$, where $r = \|x\|_{\Ho}$ and $\beta>0$.
Let $\pi_F: \Eng(n) \to \MS$ denote the \sR submersion introduced in Section \ref{sec:mag-spa}. In this section, we study the properties of this magnetic space.

 Define an isometry $\varphi_{(Q,y_0,z_0)}:\MS \to \MS$ by $\varphi_{(Q,y_0,z_0)}(x,y,z) = (Qx,y+y_0,z+z_0)$, where $Q\in SO(n)$ and $(y_0,z_0)\in \R^2$. This defines an isometric
action of $SO(n) \times \R^2$ on $\MS$.

\begin{lemma}\label{lem:plan-mag-geo-r}
Let $\MS^4$ be a four-dimensional submanifold of $\MS$ given by
$$ \MS^4 := \{ (x,y,z) \in \MS : 0 = x_3 = \dots = x_n \}. $$
If $c$ is a normal geodesic in $\MS$ with initial condition $c(0)\in \MS^4$ and $\dot{c}(0)\in T_{c(0)}\MS^4$, then $c(t)\in \MS^4$ for all $t \in \R$. Moreover, every normal  geodesic in $\MS$ has the form $c = \varphi_{(Q,y_0,z_0)} \circ c_0$, where $c_0$ is a normal  geodesic in $\MS^4$.
\end{lemma}

The proof of Lemma \ref{lem:plan-mag-geo-r} follows the same reasoning as in Lemma \ref{lem:planar-an-harmonic} and implies that understanding the geodesics in $\MS^4$ suffices to describe the geodesics in $\MS$.

Define the three-dimensional family
$$Pen_V := \{ V_{(\nu,\ell)}(r) = \frac{\ell^2}{r^2} + G_\nu^2(r) : G_\nu(r) \in Pen_F \}.$$
We call $Pen_V$ the pencil of effective potentials associated with $F$.

As in Remark \ref{rem:bij-a-vect-valued}, a pair $(\nu,\ell)$ determines an effective potential $V_{(\nu,\ell)}$.

\begin{Prop}\label{prop:mag-geo-Delta-r}
Consider the map $\Delta:\mathcal{C} \times \mathcal{I} \to [0,\infty)\times \R^2$ given by Eq. \eqref{eq:def-Delta}. The expression $\Delta (c,\mathcal{T})$ can be rewritten in terms of the effective potential $V_{(\nu,\ell)}(r)$ as follows:
\begin{equation*}\label{eq:Del-t-r}
\Delta (c,\mathcal{T}) = (\int_{r(\mathcal{T})} \frac{dr}{\sqrt{1-V_{(\nu,\ell)}(r)}},\int_{r(\mathcal{T})}\ \frac{G_{\nu}(r)dr}{\sqrt{1-V_{(\nu,\ell)}(r)}},\int_{r(\mathcal{T})} \frac{ G_{\nu}(r)F(r)dr}{\sqrt{1-V_{(\nu,\ell)}(r)}}) .
\end{equation*}
Similarly, consider the map $\operatorname{Cost}:\mathcal{C} \times \mathcal{I} \to [0,\infty)\times \R$ from Eq. \eqref{eq:def-cost-f}. The expression \(\operatorname{Cost}(c,T)\) can be written as follows:
\begin{equation*}
\begin{split}
 \operatorname{Cost}(c,\mathcal{T}) =  (\int_{r(\mathcal{T})} \frac{1-G_{\nu}(r)}{\sqrt{1-V_{(\nu,\ell)}(r)}}dr,  \int_{r(\mathcal{T})} \frac{(1-F(r))G_{\nu}(r)}{\sqrt{1-V_{(\nu,\ell)}(r)}}dr) . \\        
\end{split} 
\end{equation*} 
\end{Prop}
The proof of Proposition \ref{prop:mag-geo-Delta-r} is similar to that of Proposition \ref{prop:mag-geo-Delta-C}. 

Two geodesics related by an isometry \(\phi_{(Q,y_0,z_0)}\) belong to the same equivalence class. Using the radial Hill interval, the period map for homoclinic geodesics in $\MS^4$ is given by
\begin{equation*}
\begin{split}
\Theta([c])& = (\Theta_1([c]),\Theta_2([c]))  := 2 ( \int_{\mathcal{R}} \frac{1-G_{\nu}(r)}{\sqrt{1-G_\nu^2(r)}} dr,  \int_{\mathcal{R}} \frac{G_{\nu}(r)(1-F(r))}{\sqrt{1-G_\nu^2(r)}} dr),\\
\end{split}
\end{equation*}
where $\mathcal{R}$ is the radial Hill interval.

It suffices to consider the effective potential $V_{(\nu,0)}(r) = G_\nu^2(r)$, since Lemma \ref{lemma:equi-point-an-har} shows that homoclinic orbits exist only when $\ell = 0$. 

\subsubsection{Upper Bound for the Cut Time}

\begin{figure}%
    \centering
    {{\includegraphics[width=2.5cm]{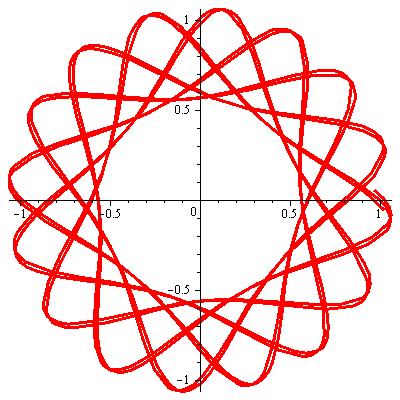} }}
    \quad
    {{\includegraphics[width=2.5cm]{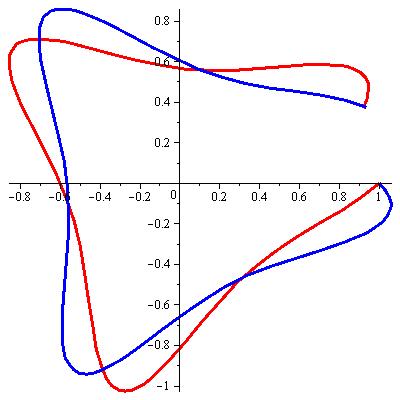} }}%
    \quad
    {{\includegraphics[width=2.5cm]{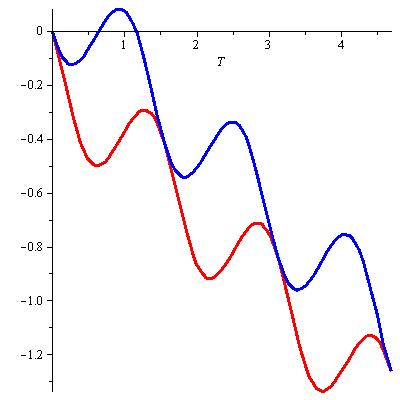} }}%
     \quad
     {{\includegraphics[width=3cm]{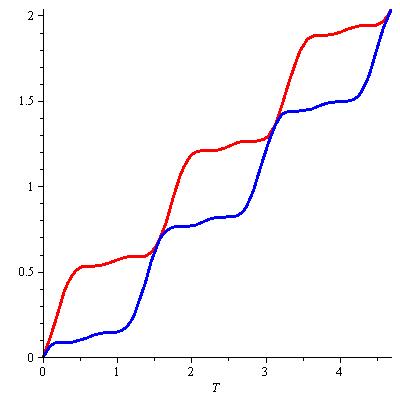} }}%
    \caption{The first panel displays a typical $r$-periodic solution of an anharmonic oscillator. The last three panels show the Maxwell point in the \((x_1,x_2)\)-, \((t,y)\)-, and \((t,z)\)-planes, respectively.}    
\end{figure}

\begin{Prop}\label{Maxwell-r}
Let $c$ be an $r$-periodic geodesic in $\MS^4$. If \(L(c)\) denotes its \(r\)-period, then \(t_{\mathrm{cut}}(c)\le L(c)\).
\end{Prop}
\begin{proof}
    Consider an $r$-periodic geodesic $c$ in $\MS^4$ with momentum $(\nu,\ell)$ and  a radial Hill interval $\mathcal{R}$ corresponding to $V_{(\nu,\ell)}\in Pen_{V}$.  Proposition \ref{prop:mag-geo-Delta-r} implies that the radial period is given by
\begin{equation}
    L(c) = 2\int_{\mathcal{R}} \frac{dr}{\sqrt{1-V_{(\nu,\ell)}(r)}},\quad\text{where}\quad r([0,\frac{1}{2}L(c)]) = \mathcal{R}.
\end{equation}
Moreover, the Hamilton equation for $\dot{\theta}$, together with the change of variables \(t=t(r)\), implies that the change in the angular coordinate after one period \(L(c)\) is
\begin{equation}\label{eq:theta-period}
    \Delta \theta(c,L(c)) =  \int_{\mathcal{R}} \frac{2\ell dr}{r^2\sqrt{1-V_{(\nu,\ell)}(r)}}.
\end{equation}

  Write $c(0) = (r_i,\theta_i,y_i,z_i)$. 
  Suppose that \(r_i\) lies in the interior of the radial Hill interval $\mathcal{R}$. In that case, there are exactly two magnetic geodesics passing
  through $c(0)$ with momentum parameters $(\nu,\ell)$ and opposite initial radial velocities, namely, the given one $c(t)$  and 
  $$\tilde{c}(t) = (\tilde{r} (t), \tilde{\theta} (t) , \tilde{y} (t) , \tilde{z}  (t)) \quad \text{characterized by}\quad \dot r (0) = - \dot {\tilde{r}}(0).$$     
  The time reversibility of the radial equation implies that $\tilde{r} (t) = r (-t)$ for all $t$. If $L:=L(c)$, then $r$-periodicity implies that $r (L) = \tilde{r} (L) = r_i$.   Proposition \ref{prop:mag-geo-Delta-r} implies
  $$ \Delta y (c,L(c)) = \Delta y (\tilde{c},L(\tilde{c})),\;\text{and}\; \Delta z (c,L(c)) = \Delta z (\tilde{c},L(\tilde{c})) .$$  
  Eq. \eqref{eq:theta-period} shows $\Delta \theta (c,L(c)) = \Delta \theta (\tilde{c},L(\tilde{c}))$.  Then, 
$$ c(L(c)) =  \big(r_i,\theta_i + \Delta \theta (c,L(c)) , y_i + \Delta y (c,L(c)) , z_i + \Delta z (c,L(c))\big) =  \tilde{c}(L(\tilde{c})).  $$ 
Since \(c\) and \(\tilde c\) are distinct arc-length-parameterized geodesics with the same endpoints at time \(L(c)\), \(L(c)\) is a Maxwell time. Hence $t_{cut}(c)  \le L(c)$.

If $r_i$ lies on the boundary of the radial Hill interval $\mathcal{R}$, Proposition \ref{prop:conj-point} applies. 
\end{proof}

\subsection{Proof of Theorem \ref{the:main-1}}\label{sub-sec:proof-theo-main-1}
We follow the framework presented in Section \ref{sec:mag-spa}. Let $\gamma_h$ denote a candidate metric line in $\Eng(n)$. In this case, $\gamma_h$ is an $r$-homoclinic geodesic. We prove the following theorem.

\begin{Theorem}\label{the:1}
Let \(\gamma_h\) be an \(r\)-homoclinic geodesic in $\Eng(n)$ whose momentum determines the polynomial $F(x)$. Consider the magnetic space $\MS$ and the projection $ \pi_{F}: \Eng(n) \to \MS$. Then $c_h := \pi_{F}\circ \gamma_h$ is a metric line.
\end{Theorem}  

Without loss of generality, we make the following simplifications. By a Carnot dilation and a translation, we may assume that $F(r) = 1 -2r^2$. By Lemma \ref{lem:plan-mag-geo-r}, it suffices to consider $c_h$ in $ \MS^4$. Using an isometry and a time translation, we may assume that $c_h(0) = (1,0,0,0)$. Therefore, by construction $c_h(t) = (x_1(t),0,y(t),z(t))$ for all $t \in \R$. Furthermore, the time reversibility of the reduced Hamiltonian $H_{\mu}$ implies $x_1(t) = x_1(-t)$ for all $t$, so $\Delta x_1(c_h,[-n,n]):= x_1(n)-x_1(-n) = 0$ for all $n \in \mathbb{N}$. 


\subsubsection{Setup}
We prove Theorem \ref{the:1} by applying Theorem \ref{the:metrtic-lines-method}. We verify that the conditions of Theorem \ref{the:metrtic-lines-method} are satisfied. The following lemma verifies condition (\ref{cond-1}). 

\begin{lemma}\label{lem:cond-1}
Under the setup of this section, $\Theta_1([c_h])$ is finite. 
\end{lemma}
\begin{proof}
By construction, the class $[c_h]$ corresponds to the polynomial $F(r)= 1-2r^2$ and the radial Hill interval $\mathcal{R} = [0,1]$. By the definition of the period map, it follows that
\begin{equation*}
    \Theta_1([c_h]) = 2\int_0^1 \frac{1-F(r)}{\sqrt{1-F^2(r)}} dr = 2\int_0^1 \frac{r}{\sqrt{1-r^2}} dr = 2. 
\end{equation*}
\end{proof}

The following lemma verifies condition (\ref{cond-2}) from Theorem \ref{the:metrtic-lines-method}.
\begin{lemma}\label{lem:geod-com-Ent}
The space $M^{4}_{F}$ with $F(r) = 1 - 2 r^2$ is geodesically compact.
\end{lemma}

We prove Lemma \ref{lem:geod-com-Ent} in Appendix \ref{ap:proof-geod-com-Ent}.

The following lemma verifies condition (\ref{cond-3}). 

\begin{lemma}\label{c}
Let $T^*_h$ be the time given by Lemma \ref{cor:T_h}. If $T^*_h<n$, then $c_n$ is not a line-geodesic, and no subsequence of $\{c_n\}$ converges to a line-geodesic. 
\end{lemma}

\begin{proof}
    By construction,  the geodesic $c_h$ satisfies $\Delta x (c_h,[-n,n]) = 0$ for all $n$. Thus, if \(c_n\) is a line-geodesic connecting \(c_h(-n)\) and \(c_h(n)\), then \(c_n\) must be a vertical line.  Furthermore, Lemma \ref{cor:T_h} implies that $\Delta y(c_h,[-n,n])>0$ for all $n>T^*_h$. Therefore, if $c_n$ is a line-geodesic, it must be a vertical line with the polynomial $G_{\nu}(r) = 1$. It follows that $\operatorname{Cost}_y(c_n,\mathcal{T}_n) >0$,  since $ G_{\nu}(r)(1-F(r))\geq 0$. However,  Eq. \eqref{eq:shot-con} implies that $\operatorname{Cost}_y(c_n,\mathcal{T}_n) <0$ for all sufficiently large $n$, since  $\Theta_2([c_h])<0$. The inequality \(\Theta_2([c_h])<0\) follows from direct integration by parts:
    \begin{equation*}
        \begin{split}
            \int_{[0,1]} \frac{r^2(1-2r^2)dr}{\sqrt{1-(1-2r^2)^2}} & = \frac{r}{4} \sqrt{1-(1-2r^2)^2} \Big|_0^1 - \frac{r}{4} \int_{[0,1]}  \sqrt{1-(1-2r^2)^2} dr \\
            & = - \frac{r}{4} \int_{[0,1]}  \sqrt{1-(1-2r^2)^2} dr < 0.
        \end{split}
    \end{equation*}
    Therefore, the vertical line does not satisfy the endpoint condition in Eq. \eqref{eq:shot-con}. Moreover, $c_n$ cannot converge to a vertical line-geodesic $c$ with $x=0$ since $\operatorname{Cost}_y(c,\mathcal{T}) = 0$ for every time intervals $\T$. Thus, the vertical line $c$ with $x=0$ does not satisfy the asymptotic condition in Eq. \eqref{eq:asymp-cond-h}.     
\end{proof}

The following lemma characterizes the $r$-homoclinic geodesics in $\MS^4$ and establishes condition (\ref{cond-4}).

\begin{lemma}\label{lem:theta-r-bounded}
Under the identification of an equivalence class with its pair \((G_\nu,R)\), the set $Homc^\pm$ of equivalence classes of \(r\)-homoclinic geodesics is given by 
$$Homc^{\pm}  := \{ (G_\nu,\mathcal{R}_{\beta}): \; G_{\nu}(r) = \pm(1-\beta r^2) \; \text{and}\; \mathcal{R}_{\beta} = [0,\sqrt{\frac{2}{\beta}}] \; \text{with}\; \beta >0\}.$$
 Moreover, the map $\Theta_2: Homc^{\pm}\to \R$ is one-to-one.
\end{lemma}

\begin{proof}
Lemma \ref{lemma:equi-point-an-har} establishes the first statement. It remains to prove that $\Theta_2: Homc^{\pm}\to \R$ is one-to-one. Consider the class $[c_\beta]$ corresponding to the polynomial $G_{\nu}(r) = \pm(1-\beta r^2)$. Applying the change of variable $r = \sqrt{\frac{2}{\beta}} \tilde{r}$, we obtain
\begin{equation*}
\begin{split}
\Theta_2([c_{\beta}]) & = \pm4\int_{\mathcal{R}_{\beta}} \frac{r^2(1-\beta r^2)dr}{\sqrt{1-(1-\beta r^2)^2}} = \pm4\big( \sqrt{\frac{2}{\beta}} \big)^3\int_{[0,1]} \frac{\tilde{r}^2(1-2\tilde{r}^2)d\tilde{r}}{\sqrt{1-(1-2\tilde{r}^2)^2}}. \\
                      & = \pm\big( \sqrt{\frac{2}{\beta}} \big)^3 \Theta_2([c_h]).
\end{split}
\end{equation*}
Since $\Theta_2([c_h])\neq 0$, the map $\beta \to \Theta_2([c_{\beta}])$ is one-to-one for $\beta>0$. 
\end{proof}

\begin{proof}[Proof of Theorem \ref{the:1}]
Let $\gamma_h$ be an $r$-homoclinic geodesic in $\Eng(n)$ associated with a polynomial $F(r) = 1-\beta r^2$, and let $c_h:= \pi_{F} \circ \gamma_h$. By a Carnot dilation and a translation, as in the setup above, we may assume $F(r) = 1-2r^2$. In this setting, Lemmas \ref{lem:cond-1}, \ref{lem:geod-com-Ent}, \ref{c},  and \ref{lem:theta-r-bounded} establish conditions (\ref{cond-1}), (\ref{cond-2}), (\ref{cond-3}), and (\ref{cond-4}), respectively. Hence, Theorem \ref{the:metrtic-lines-method} implies that $c_h$ is a metric line in $\MS$.
\end{proof}

\begin{proof}[Proof of Theorem \ref{the:main-1}]
Proposition \ref{prp:no-met-lin-Eng} implies that $r$-periodic geodesics and small-oscillation geodesics are not metric lines in $Eng(n)$. Lemma \ref{lem:sR-sub-metric-line} shows that the line-geodesics are metric lines. Finally, if $\gamma$ is an $r$-homoclinic geodesic, Theorem \ref{the:1} shows that $c := \pi_F\circ \gamma$ is a metric line in $\MS$.  Therefore, Lemma \ref{lem:sR-sub-metric-line}  implies that $\gamma$ is a metric line in $Eng(n)$.
\end{proof}

\addtocontents{toc}{\protect\setcounter{tocdepth}{1}}
\subsubsection*{Conclusion}\label{sec:conclusions}
\addtocontents{toc}{\protect\setcounter{tocdepth}{2}}

We have presented an approach, called the sequence method, for proving that a particular geodesic in a metabelian Carnot group is a metric line. The sequence method projects the geodesic onto the \sR magnetic space $\MS$. Theorems \ref{the:metrtic-lines-method} and \ref{the:metrtic-lines-method-d-t} provide sufficient conditions for homoclinic and direct-type geodesics in \(M_F\) to be metric lines, respectively. In Theorem \ref{the:main-1}, we give a complete classification of the metric lines in the Engel-type group using the sequence method.

Some directions for future work are: 
\begin{itemize}
\item  The strategy from Lemma \ref{lem:geod-com-Ent} may be applied to establish geodesic compactness for magnetic spaces with integrable geodesic flow. A key next step is to establish geodesic compactness when the geodesic flow is non-integrable.

\item Provide a complete classification of metric lines in step-3 Carnot groups with a semidirect product structure, where $dim \;\Vs = 1$.

\item Prove that bounded generic geodesics arbitrarily close to a critical point are not metric lines.

\item Prove that the homoclinic geodesics in the jet space $\J$ are metric lines.
\end{itemize}

Two interesting questions are: 
\begin{itemize}
\item Are all metric lines in step-3 Carnot groups related to Euler-solitons, as in Proposition \ref{prp:Euler-Elastica-Eng}? 

\item Apart from line-geodesics, are all metric lines in step-3 Carnot groups homoclinic?
\end{itemize}

\appendix

\section{Abnormal Geodesics}\label{sub-AP:ab-geo}

\begin{proof}[Proof of Lemma \ref{lemma:abn-geo-mag}]
We compute abnormal extremals using Pontryagin’s maximum principle. Consider the following optimal control problem:
\begin{equation}\label{eq:abn-proof-gama}
\dot{c}(t) = u_1 \tilde{X}_1 + \cdots + u_n \tilde{X}_n + u_{n+1} \tilde{Y}, 
\end{equation}
with prescribed endpoints $c(0)$ and $c(T)$. It follows from Eq. \eqref{eq:abn-proof-gama} that $\dot{x}_i = u_i$. On the fixed time interval \([0,T]\), the Cauchy–Schwarz inequality reduces the length-minimization problem to the action-minimization problem
\begin{equation*}\label{eq:abn-proof-gama-2}
\int_0^T \frac{u_1^2+ \cdots +u_{n+1}^2}{2}dt \to min.
\end{equation*} 
For an abnormal extremal, the multiplier associated with the cost vanishes, so the corresponding Hamiltonian is
\begin{equation}\label{eq:abn-proof-gama-3}
H(p,x,y,z,u) = u_1 p_{x_1} + \cdots +  u_n p_{x_n} + u_{n+1}(p_y + p_z F(x)). 
\end{equation}
Pontryagin’s maximum principle yields the Hamiltonian equations for the momentum coordinates, together with the maximum and nontriviality conditions.
\begin{equation}\label{eq:abn-proof-gama-4}
\dot{p}_{x_i} = -p_zu_{n+1}(t) \frac{\partial F}{\partial x_i}\;\;\text{for} \;\;i = 1,\dots,n,\;\;\dot{p}_y = 0 \;\;\text{and}\;\;\dot{p}_z = 0.
\end{equation}
The maximum condition is defined as follows:
\begin{equation}\label{eq:pont-prin}
\max_{u \in \R^{n+1}} H(p(t),c(t),u(t)) = H(p(t),\hat{c}(t),\hat{u}(t)), 
\end{equation}
where  $\hat{c}(t) = (\hat{x}(t),\hat{y}(t),\hat{z}(t))$ denotes the optimal trajectory, subject to the nontriviality condition $p(t) \neq 0$. Therefore, the maximum condition implies the following:
\begin{equation}\label{eq:abn-proof-gama-5}
0 = p_{x_i}, \;\; \text{for} \;\;i = 1,\dots,n,\;\;\text{and}\;\; 0 = p_y + p_zF(x). 
\end{equation}
By Eq. \eqref{eq:abn-proof-gama-5} and the nontriviality condition, we obtain $(p_y,p_z) \neq (0,0)$. If $p_z = 0$, then $p_y = 0$. Thus, $p_z \neq 0$. By differentiating Eq. \eqref{eq:abn-proof-gama-5} with respect to time and using Eq. \eqref{eq:abn-proof-gama-4}, it follows that  
\begin{equation}\label{eq:abn-proof-gama-6}
0 = p_zu_{n+1}(t)  \frac{\partial F}{\partial x_i}\;\;\text{for} \;\;i = 1,\dots,n,\;\;\text{and}\;\;  0 = \frac{\partial F}{\partial x_1} u_1 + \cdots + \frac{\partial F}{\partial x_n} u_n .  
\end{equation}
Equation \eqref{eq:abn-proof-gama-6} determines two possible regimes for an abnormal extremal: $u_{n+1}(t) = 0$ or $dF= 0$. An abnormal extremal may switch between these regimes.

\textbf{(Vertical)} If $dF = 0$, then \(x(t)\) is a critical point of \(F\). Since \(F\) is a quadratic polynomial with a unique critical point \(x^*\), it follows that \(x(t)=x^*\). Hence, \(\dot x_i(t)=u_i(t)=0\) for \(i=1,\ldots,n\). Therefore, a vertical abnormal extremal is a geodesic if
and only if it is a vertical line-geodesic.

\textbf{(Horizontal isocontours)}  If $u_{n+1}(t) = 0$, then $c(t)\in \Ho$ and the second equation in \eqref{eq:abn-proof-gama-6} implies that the curve \(x\) lies in a level set of $F$. Therefore, a horizontal abnormal extremal is a geodesic if and only if it is a horizontal line-geodesic.

An abnormal extremal may switch between the two regimes described above. However, an abnormal geodesic cannot switch from the vertical regime to the horizontal-isocontour regime, or vice versa, since such a switch would produce a corner. Since a length-minimizing sub-Riemannian geodesic cannot have corners \cite{hakavuori2016non}, every abnormal geodesic remains in one of the two regimes. Therefore, every abnormal geodesic is either a vertical line-geodesic or a horizontal line-geodesic. Since both types are also normal geodesics with constant polynomial $G_\nu^F$, \(M_F\) has no strictly abnormal geodesics.

\end{proof}

\begin{proof}[Proof of Lemma \ref{lemma:abn-geo-Eng}]
Set $Y_i := [X_i,Y]$ and $Y_{n+1} := [X_i, Y_i]$ for all $i=1,\dots,n$, where $\{X_1,\dots,X_n,Y\}$ is the frame for $\D$ presented in Section \ref{sec:Eng-type}. Applying the strategy used in the proof of Lemma \ref{lemma:abn-geo-mag}, the Hamiltonian function is expressed as

\begin{equation}\label{eq:abn-proof-gama-3-En}
H(p,g,u) = u_1 P_{X_1} + \cdots +  u_n P_{X_n} + u_{n+1} P_{Y}, 
\end{equation}
where $P_{X_1},\ldots,P_{X_n}, P_{Y}$ denote the left-invariant momentum functions associated to the left-invariant vector fields tangent to $\Di$. Pontryagin’s maximum principle yields a Hamiltonian system governing the left-invariant momentum functions
\begin{equation}\label{eq:abn-proof-gama-4-En}
\begin{split}
\dot{P}_{X_i} & = -u_{n+1}(t) P_{Y_i}\;\;\text{for} \;\;i = 1,\dots,n,\;\;\dot{P}_{Y} = u_1(t) P_{Y_1} + \cdots +  u_n(t) P_{Y_n},\\
\dot{P}_{Y_i} & = u_{i}(t) P_{Y_{n+1}}\;\;\text{for} \;\;i = 1,\dots,n,\;\text{and}\;\dot{P}_{Y_{n+1}} = 0. 
\end{split}
\end{equation}
Differentiating the maximum conditions and using Eq. \eqref{eq:abn-proof-gama-4-En} yields
\begin{equation}\label{eq:abn-proof-gama-5-En}
0 = -u_{n+1}(t) P_{Y_i} , \;\; \text{for} \;\;i = 1,\dots,n,\;\;\text{and}\;\; 0 = u_1(t) P_{Y_1} + \cdots +  u_n(t) P_{Y_n}. 
\end{equation}
Equation \eqref{eq:abn-proof-gama-5-En} yields the following two regimes for an abnormal extremal.

\textbf{(Vertical)}  Suppose that $u_{n+1}(t)\neq 0$. Then Eq. \eqref{eq:abn-proof-gama-5-En} implies $P_{Y_i}=0$ for all $i=1,\ldots,n$. It follows from Eq. \eqref{eq:abn-proof-gama-4-En} that
$$0=\dot P_{Y_i}=u_i(t)P_{Y_{n+1}},
\qquad i=1,\ldots,n.$$
By the maximum conditions, \(P_{X_i}=P_Y=0\). Therefore, if \(P_{Y_{n+1}}=0\) as well, the entire covector would vanish, contradicting the nontriviality condition. Thus $P_{Y_{n+1}}\neq 0$. Consequently, $u_i(t)=0$ for $i=1,\ldots,n$. Hence, in the vertical regime, the abnormal extremal is tangent to $\mathrm{span}\{ Y_{\gamma(t)}\}$. Therefore, a vertical abnormal extremal is a geodesic if and only if it is a vertical line-geodesic.

\textbf{(Horizontal)} Suppose that $u_{n+1}(t)= 0$. The abnormal extremal is horizontal. Since $\Ho$ is a Euclidean space, its geodesics are straight lines. Therefore, a horizontal abnormal extremal is a geodesic if and only if it is a horizontal line-geodesic.

Since both vertical lines and horizontal line-geodesics are also normal geodesics, \(\operatorname{Eng}(n)\) has no strictly abnormal geodesics.
\end{proof}

\section{Sequence Method II}\label{ap:se-me}

\subsection{Sequence Method for Direct-Type Geodesics}
The following theorem addresses direct-type metric lines. 
\begin{Theorem}[Sequence method for direct-type geodesics]\label{the:metrtic-lines-method-d-t}
Let $\G$ be a \ma Carnot group with a semidirect product structure such that $\dim \; \Vs = 1$. Let $\gamma_d$ be a direct-type geodesic in $\G$ whose momentum $\mu$ determines a polynomial $F$, and let $c_d:= \pi_F\circ\gamma_d$ be the geodesic in $\MS$. Without loss of generality, assume that $[c_d]\in Hetc^+(x_0,x_1)$. Assume the following conditions hold:
\begin{enumerate}
\item The value $\Theta_1([c_d])$ is finite. 
\item  There exists a ball $B_{\Ho}(r,x_0)\subset\Ho$ such that the region $B_{\Ho}(r,x_0) \times \R^{2}$ is geodesically compact.

\item  There exists a sequence of precisely normal minimizing geodesics $\{c_n\}$ joining the points $c_d(-n)$ and $c_d(n)$. Moreover, no subsequence of $\{c_n\}_{n\in\mathbb N}$ converges to a line-geodesic. 

\item The map $\Theta_2: Hetc^+(x_0,x_1) \to \R$ is one-to-one.

\end{enumerate}
Then $c_d$ is a metric line in $\MS$. 
\end{Theorem}

The proof of Theorem \ref{the:metrtic-lines-method-d-t}  follows the same strategy and is therefore omitted. See \cite{bravo2022metric} for the case of jet spaces.

\subsection{Complementary Results for the Sequence Method for Homoclinic Geodesics}

Consider the vector space of polynomials on $\Ho$. Define  $\|F\|_{\sup} :=  \sup_{x \in B_{\Ho}(1,0)} |F(x)|$ to be the sup norm, and let
$$\mathcal{P}(k):=\{F: deg\;F\leq k,\;\;\|F\|_{\sup} \leq 1 \}. $$ 

It is well known that $\mathcal{P}(k)$ is compact.

\subsubsection{Proof of Lemma \ref{lem:geod-com-Ent} } \label{ap:proof-geod-com-Ent}

\begin{proof}[Proof of Lemma \ref{lem:geod-com-Ent}:]

Consider the ball $B_{\Ho}(r^*,0)\subset\Ho$, where $r^*>0$ is arbitrary. Let $\{(c_n,\mathcal{T}_n)\}\subset \operatorname{Com}(r^*,0,C)$, where $c_n(t) = (x_n(t),y_n(t),z_n(t))$ is a geodesic in $\MS$ with $F(r) = 1 - 2r^2$, and suppose $\mathcal{T}_n$ exhausts $\R$ as $n \to \infty$. We first show that the radial intervals $r(\mathcal{T}_n)$ are uniformly bounded and then show that $s(\mathcal{T}_n)$ is uniformly bounded.

We prove the first part by contrapositive. Specifically, if $x_n(\mathcal{T}_n)$ is unbounded, then $\operatorname{Cost}(c_n,\mathcal{T}_n)$ is also unbounded. 
The sequence of geodesics $c_n$ induces a sequence of pairs $(G_n,\ell_n)$. We first show that a convergent subsequence $(G_{_j},\ell_{n_j})$ can be extracted. Write $r(\mathcal{T}_n) := [r^{min}_n,r^{max}_{n}]$, and consider the affine map $h_n(\tilde{r}) = r^{min}_n + u_n \tilde{r}$ where $u_n := r^{max}_n - r^{min}_n$. The map $h_n$ sends $~[0,1]$ to $r(\mathcal{T}_n)$. Since the effective potential satisfies $ V_n(r) = \frac{\ell_n^2}{r^2} + G^2_n(r) \leq 1$, it follows that  $\tilde{G}_n(\tilde{r}) := G_n(h_n(\tilde{r}))\in \mathcal{P}(4)$. Moreover, if $x_n(t)$ leaves the ball $B_{\Ho}(r^*,0)$ for some $t\in \mathcal{T}_n$, then  $|\ell_n| < r^*$. Hence, $\{\ell_n\}$ is bounded. Therefore, there exists a subsequence $(\tilde{G}_{n_j},\ell_{n_j})$ converging to $(\tilde{G},\ell)$. We consider two cases: $\tilde{G} \not\equiv 1$ and $\tilde{G} \equiv 1$.

Case $\tilde{G}(\tilde{r}) \not\equiv1$. Applying the change of variables $r = h_{n_j}(\tilde{r})$, it follows that
$$ \operatorname{Cost}_t(c_n,\mathcal{T}_n) = u_n \int_{[0,1]} \frac{1-\tilde{G}_{n_j}(\tilde{r})}{\sqrt{1-\frac{\ell^2_{n_j}}{h_{n_j}^2(\tilde{r})} - \tilde{G}^2_{n_j}(\tilde{r})}} d\tilde{r} , $$
Therefore, Fatou's lemma implies
$$ 0 <\int_{[0,1]} \frac{1-\tilde{G}(\tilde{r})}{\sqrt{1-\frac{\ell^2}{h_{n_j}^2(\tilde{r})} - \tilde{G}^2(\tilde{r})}} d\tilde{r} \leq \liminf_{n_j \to \infty}  \int_{[0,1]} \frac{1-\tilde{G}_n(\tilde{r})}{\sqrt{1-\frac{\ell^2_{n_j}}{h_{n_j}^2(\tilde{r})} - \tilde{G}^2_{n_j}(\tilde{r})}} d\tilde{r} .$$
Since $u_{n_j} \to \infty$, it follows that  $\operatorname{Cost}_t(c_{n_j},\mathcal{T}_{n_j}) \to \infty$ as $n_j \to \infty$.

Case  $\tilde{G}(\tilde{r}) = 1$. There exists $n^*\in\mathbb{N}$ such that $\tilde{G}_{n_j}(\tilde{r}) > \frac{1}{2}$ for all $\tilde{r}\in [0,1]$ if $n_j > n^*$. Hence, $ 1-\frac{\ell_{n_j}^2}{r^2} - G_{n_j}^2(\tilde{r})\leq \frac{3}{4} $. Given the expression $1-F(r) = 2r^2$, it follows that
\begin{equation*}
\begin{split}
\frac{2 u_{n_j}}{\sqrt{3}}  & \int_{[0,1]} r^2dr <  \operatorname{Cost}_y(G_{n_j},\mathcal{T}_{n_j}) . \\
\end{split}
\end{equation*} 
So $ \operatorname{Cost}_y(G_{n_j},\mathcal{T}_{n_j})\to \infty$ as $n_j \to \infty$. Therefore, the radial intervals  $r(\mathcal{T}_n)$ are uniformly bounded; that is, there exists $C_r>0$ such that 
$$ r(\mathcal{T}_n)\subset  [0,C_r]\qquad \text{for all}\;\;n\in \mathbb{N}.$$

If $\ell_n = 0$, Proposition \ref{Maxwell-r} implies $s(\mathcal{T}_n)<2C_r$. If $\ell_n \neq 0$ and $r_n(t)$ is constant, i.e., it is a relative equilibrium point by Lemma \ref{lemma:equi-point-an-har}, then  Proposition \ref{Maxwell-r} implies $s(\mathcal{T}_n)<2\pi C_r$. 

Suppose now that $\ell_n \neq 0$ and $r_n(t)$ is not constant. Consider the one-form  $\phi_n$ defined by
$$ \phi_n =\frac{2\ell_ndr}{r^2\sqrt{1-V_n(r)}}.\;\text{It follows that}\;\Delta\theta(c_n,\mathcal{T}_n) = \int_{r(\mathcal{T}_n)} \phi_n. $$

Define the Hamiltonian function
$$H_{n}(p_r,r) := \frac{1}{2} \big(p_r^2 + \frac{\ell_n^2}{p_r^2} + G_n^2(r)\big).$$
The one-form $\phi_n$ is
defined on a simple closed curve given by
$$ \alpha_n  := \{(p_r,r) \in \R^2: \frac{1}{2} = H_{n}(p_r,r),\;\text{and}\; r \in \mathcal{R}_n   \}.  $$
Thus, $\phi_n$ is a smooth, closed, and non-exact one-form. Indeed, by the preimage theorem, the curve $\alpha_n$ is smooth at a point $(p_r,r) \in \alpha_n$ if and only if $\nabla H_n(p_r,r) \neq 0$, which is equivalent to saying that $(p_r, r)$ is not a relative equilibrium point. It is well known that $\phi_n$ depends smoothly on the parameters $\ell_n, a_n$, and $b_n$ whenever the curve $\alpha_n$ is finite; see \cite[Remark 2.21]{bravo2022metric}, \cite[Lemma 2.24]{bravo2022metric}, or \cite[Chapter 3.10]{lawden2013elliptic}.

It remains to show that there exists a constant $C_\theta>0$ such that $|\Delta\theta(c_n,\mathcal{T}_n)|<C_\theta$ for all $n \in \mathbb{N}$. If this is true, then $s(\mathcal{T}_n)\leq 2C_r + C_rC_\theta$.

 We establish this bound in two steps. First, the smoothness of the one-form $\phi_n$ implies that $h^*_n\phi_n$ is smooth, where $h_n^*$ denotes the pullback by $h_n$. The lower limit of integration, $h_n^{-1}(r(0))$, is fixed, whereas the upper limit, $h_n^{-1}(r(t))$, varies with $t \in \mathcal{T}_n$. The integral of $h^*_n\phi_n$ is a smooth function with respect to the upper limit $h_n^{-1}(r(t))$ and attains its maximum
$$ \max_{h_n^{-1}(r(t)) \in [0,1]} \big|\int_{h_n^{-1}([0,t])} h^*_n\phi_n\big|.$$
Second, this maximum depends smoothly on the parameters $\ell_n, a_n$, and $b_n$. Since these parameters belong to a sequentially compact set, there exists $C_\theta>0$ such that $|\Delta\theta(c_n,\mathcal{T}_n)|<C_\theta$ for all $n \in \mathbb{N}$.

\end{proof}

\bibliographystyle{plain}
\bibliography{Bibliography_corrected}

\end{document}